\RequirePackage{silence}
\documentclass[11pt]{article}

\usepackage[margin= 2.0 cm, bottom=20mm,footskip=12mm,top=20mm]{geometry}

\usepackage{amsthm}
\usepackage{amsmath}
\usepackage{amssymb}
\usepackage{setspace}
\usepackage{mathtools}
\usepackage{verbatim}
\usepackage{csquotes}
\usepackage{graphicx}
\usepackage[hidelinks]{hyperref}
\usepackage{thm-restate}
\usepackage{cleveref}
\usepackage{graphicx}
\usepackage{enumitem}
\usepackage{framed}
\usepackage{subcaption}
\usepackage{floatrow}
\usepackage[T1]{fontenc}
\usepackage{tikz}
\usepackage{mathdots}
\usepackage{xcolor}
\usepackage{diagbox}
\usepackage{colortbl}
\usepackage[absolute,overlay]{textpos}
\usetikzlibrary{calc}
\usetikzlibrary{decorations.pathreplacing}
\usetikzlibrary{positioning,patterns,decorations.pathmorphing}
\usetikzlibrary{arrows,shapes,positioning}
\usetikzlibrary{decorations.markings}
\def \vx {circle[radius = .12][fill = black]}

\def \smvx {circle[radius = .08][fill = black]}

\tikzstyle{edge}=[very thick]
\definecolor{bostonuniversityred}{rgb}{0.8, 0.0, 0.0}
\definecolor{arsenic}{rgb}{0.23, 0.27, 0.29}
\tikzstyle{diredge}=[postaction={decorate,decoration={markings,
		mark=at position .95 with {\arrow[scale = 1]{stealth};}}}]

\tikzstyle{spring}=[thick,decorate,decoration={zigzag,pre length=0.2cm,post
  length=0.2cm,segment length=6}]

\newcommand{\defPt}[3]{
	\def \pt {(#1, #2)}
	\coordinate [at = \pt, name = #3];
}

\newcommand{\defPtm}[2]{

	\coordinate [at = #1, name = #2];
}

\newcommand{\fitellipsis}[2] 
{\draw [fill=green]let \p1=(#1), \p2=(#2), \n1={atan2(\y2-\y1,\x2-\x1)}, \n2={veclen(\y2-\y1,\x2-\x1)}
    in ($ (\p1)!0.5!(\p2) $) ellipse [ x radius=\n2/2+0cm, y radius=0.1cm, rotate=\n1];
}

\theoremstyle{plain}

\newtheorem*{thm*}{Theorem}
\newtheorem{thm}{Theorem}
\Crefname{thm}{Theorem}{Theorems}

\newtheorem*{lem*}{Lemma}
\newtheorem{lem}[thm]{Lemma}
\Crefname{lem}{Lemma}{Lemmas}

\newtheorem*{claim*}{Claim}
\newtheorem{claim}{Claim}
\crefname{claim}{Claim}{Claims}
\Crefname{claim}{Claim}{Claims}

\Crefname{prop}{Proposition}{Propositions}

\crefname{cor}{Corollary}{Corollaries}

\newtheorem{conj}[thm]{Conjecture}
\crefname{conj}{Conjecture}{Conjectures}

\Crefname{qn}{Question}{Questions}

\Crefname{obs}{Observation}{Observations}

\newtheorem{ex}[thm]{Example}
\Crefname{ex}{Example}{Examples}

\theoremstyle{definition}

\Crefname{prob}{Problem}{Problems}

\Crefname{defn}{Definition}{Definitions}

\newtheorem*{defn*}{Definition}

\theoremstyle{remark}

\newcommand{\floor}[1]{
    \left\lfloor #1 \right\rfloor
}

\expandafter\def\expandafter\normalsize\expandafter{%
    \normalsize
    \setlength\abovedisplayskip{4pt}
    \setlength\belowdisplayskip{4pt}
    \setlength\abovedisplayshortskip{4pt}
    \setlength\belowdisplayshortskip{4pt}
}

\usepackage[square,sort,comma,numbers]{natbib}
 \setlist[itemize]{leftmargin=*}

\newcommand{\tbpp}{top-bottom path-pair}
\newcommand{\tbpps}{top-bottom path-pairs}
\newcommand{\eps}{\varepsilon}

\DeclareFontFamily{OT1}{pzc}{}
\DeclareFontShape{OT1}{pzc}{m}{it}{<-> s * [1.10] pzcmi7t}{}
\DeclareMathAlphabet{\mathpzc}{OT1}{pzc}{m}{it}

\title{\vspace{-0.9cm} Cycles of many lengths in Hamiltonian graphs}

\author{
Matija Buci\'c\thanks{Department of Mathematics, ETH, Z\"urich, Switzerland. Email: \href{mailto:matija.bucic@math.ethz.ch} {\nolinkurl{matija.bucic@math.ethz.ch}}.}
 \and
Lior Gishboliner\thanks{Department of Mathematics, ETH, Z\"urich, Switzerland. Email: \href{mailto:lior.gishboliner@math.ethz.ch} {\nolinkurl{lior.gishboliner@math.ethz.ch}}.}
\and
Benny Sudakov\thanks{Department of Mathematics, ETH, Z\"urich, Switzerland. Email:
\href{mailto:benjamin.sudakov@math.ethz.ch} {\nolinkurl{benjamin.sudakov@math.ethz.ch}}.
Research supported in part by SNSF grant 200021\_196965.}
}

\date{}
\begin{document}

\maketitle

\begin{abstract}
In 1999, Jacobson and Lehel conjectured that for $k \geq 3$, every $k$-regular Hamiltonian graph has cycles of at least linearly many different lengths. This was further strengthened by Verstra\"{e}te, who asked whether
the regularity can be replaced with the weaker condition that the minimum degree is at least $3$. 
Despite attention from various researchers, until now, the best partial result towards both of these conjectures was a $\sqrt{n}$ lower bound on the number of cycle lengths. We resolve these conjectures asymptotically, by showing that the number of cycle lengths is at least $n^{1-o(1)}$.
\end{abstract}

\section{Introduction}
The study of cycles in graphs goes back to the early days of graph theory and has been fundamental ever since.
Of particular interest are Hamilton cycles, i.e. cycles passing through all the vertices of a graph. 
Starting with the cornerstone theorem of Dirac \cite{Dirac}, there are many results giving sufficient conditions for a graph to be Hamiltonian, for some other classical examples see \cite{Bondy1980,Chvatal,CE,Fan,Ore}. In 1973, Bondy \cite{Bondy_conjecture} made the ``meta-conjecture'' that any non-trivial condition which guarantees the existence of a Hamilton cycle, should also guarantee that the given graph is pancyclic, i.e. contains cycles of all possible lengths, with possibly a simple family of exceptions. This assertion turned out to be influential, and by now there are numerous appealing results of this type. For example, Bondy himself \cite{Bondy_pancyclic} proved that Ore's sufficient condition for Hamiltonicity (that the sum of degrees of any pair of non-adjacent vertices is at least $n$), implies that the graph is either pancyclic or isomorphic to the complete bipartite graph $K_{n/2,n/2}$. Bauer and Schmeichel \cite{BS}, relying on previous results of Schmeichel and Hakimi \cite{SH}, have shown that the sufficient conditions for Hamiltonicity of Bondy \cite{Bondy1980}, Chv\'atal \cite{Chvatal} and Fan \cite{Fan} all imply pancyclicity, barring a small family of exceptions. Jackson and Ordaz \cite{JO} conjectured that any graph $G$ whose connectivity $\kappa(G)$ is strictly larger than its independence number $\alpha(G)$ must be pancyclic. This conjecture is motivated by the classical theorem of Chv\'atal and Erd\H{o}s \cite{CE} that a graph with $\kappa(G) \geq \alpha(G)$ must be Hamiltonian. An approximate form of the conjecture was proven by Keevash and Sudakov \cite{KS}, who showed that $\kappa(G) \geq 600 \alpha(G)$ is already sufficient for pancyclicity. 

Pancyclicity is just an instance of a wider class of problems, which study the properties of the set of cycle lengths of a graph with connection to other graph parameters. The set of cycle lengths of $G$ is called its {\em cycle spectrum}, and denoted $\mathcal{C}(G)$.
 There are by now numerous results relating properties of $\mathcal{C}(G)$ to various graph parameters. For example, Erd\H{o}s \cite{Erdos1992} conjectured that a graph $G$ with girth $g$ and average degree $d$ must satisfy $|\mathcal{C}(G)| \geq \Omega(d^{\lfloor \frac{g-1}{2} \rfloor})$. The case $g = 5$ was settled by Erd\H{o}s, Faudree, Rousseau, and Schelp \cite{EFRS}. Later, Sudakov and Verstra\"ete \cite{SV} proved the full conjecture in a strong form. Another example is a result of Gould, Haxell and Scott \cite{GHS} that a graph with minimum degree $cn$ must have a cycle of any even length between $4$ and $ec(G) - K$, where $ec(G)$ is the length of a longest even cycle in $G$ and $K$ is a constant depending only on $c$. We should also mention the recent work of Gao, Huo, Liu and Ma \cite{GHLM}, who proved several conjectures relating properties of $\mathcal{C}(G)$ to the minimum degree, connectivity or chromatic number of $G$.

Bondy's meta-conjecture is about conditions for Hamiltonicity which imply pancyclicity. 
A natural question in the opposite direction is as follows: Let us {\em assume} that a graph $G$ is Hamiltonian; under which assumptions can we also guarantee that $G$ is pancyclic? Since pancyclicity is sometimes too strong of a requirement, we can relax it and ask to find many cycle lengths. 
Questions of this type were first introduced by Jacobson and Lehel at the 1999 conference ``Paul Erd\H{o}s and His Mathematics'', where they asked for the minimum size of the cycle spectrum of a $k$-regular Hamiltonian graph $G$ on $n$ vertices?  
The aforementioned result of Bondy \cite{Bondy_pancyclic} implies that if $k = \lceil n/2 \rceil$, then $G$ is pancyclic unless $G = K_{n/2,n/2}$. At the other extreme, if $k = 2$ then $G$ clearly has just one cycle. 
Jacobson and Lehel conjectured that already for $k \geq 3$, the number of cycle lengths should be linear in $n$. This is best possible, since 
they also observed that one cannot expect to have pancyclicity. Indeed, assuming $2k$ divides $n$, take $\frac{n}{2k}$ disjoint copies of $K_{k,k}$, ordered in a cycle, remove an edge from each of them, and add an edge between any two consecutive copies such that the resulting graph is $k$-regular. It is not hard to see that in this construction, the possible cycle lengths are precisely all the even integers between $4$ and $2k$ and between $\frac{2n}{k}$ and $n$. This gives in total $\frac{n}{2} \cdot \frac{k-2}{k} + k$ different lengths. 

Soon after the above question was first circulated, Gould, Jacobson and Pfender proved that $|\mathcal{C}(G)| \geq \Omega(\sqrt{n})$ for every $k$-regular $n$-vertex Hamiltonian graph $G$ (with $k \geq 3$). This bound was subsequently obtained by several other authors. Yet, prior to our work, the $\sqrt{n}$ bound was the best known result. In particular, Gir{\~a}o, Kittipassorn and Narayanan \cite{GKN} remarked that improving this estimate would be of considerable interest. 
Furthermore, the following strengthening of the above conjecture of Jacobson and Lehel, which replaces the $k$-regularity condition with the assumption that the minimum degree is at least $3$, was proposed by 
Verstra\"ete \cite{Verstraete}.
\begin{conj}\label{conj:main}
An $n$-vertex Hamiltonian graph $G$ with $\delta(G) \geq 3$ has $\Omega(n)$ different cycle lengths. 
\end{conj}

While the special case of this conjecture for regular graphs already seems quite challenging, it is natural to expect that the full Conjecture \ref{conj:main} is even harder. The reason for this is that often problems become more difficult when the regularity requirement is replaced by a minimum degree assumption. 
One well-known example is a conjecture of Thomassen \cite{Thomassen_girth}, that a graph with a large enough minimum degree contains a subgraph of large minimum degree and large girth. This conjecture is open even for girth $7$. However, this statement becomes easy if the given graph is regular, see e.g. \cite{MPS}. 
Such situations arise also for questions related to the one studied here: A classical result of Smith (see \cite{Tutte} and also \cite{Thomassen_300}) states that every Hamiltonian $3$-regular graph $G$ contains a second Hamilton cycle. As was shown by Entringer and Swart \cite{ES}, this is no longer true if instead of $3$-regularity we assume that $\delta(G) \geq 3$ (even if all degrees are equal to $3$ or $4$). 
Gir{\~a}o, Kittipassorn and Narayanan \cite{GKN} required an involved proof to even show that a Hamiltonian $G$ with $\delta(G) \geq 3$ contains a second cycle of length at least $n - o(n)$. In contrast, for regular $G$, this proof can be simplified considerably and gives a better bound. 

It is worth noting that if one replaces the minimum degree requirement $\delta(G) \geq 3$, with the requirement that the average degree is at least $3$, then the aforementioned lower bound of $\Omega(\sqrt{n})$ is tight. More generally, Milans, Pfender, Rautenbach, Regen and West \cite{MPRRW} have shown that a graph $G$ with $n$ vertices and $m$ edges satisfies $|\mathcal{C}(G)| \geq (1-o(1))\sqrt{m - n}$, and this is tight.

In this paper, we prove the following theorem, which resolves Conjecture \ref{conj:main} asymptotically:
\begin{thm}\label{thm:main}
An $n$-vertex Hamiltonian graph $G$ with $\delta(G) \geq 3$ contains cycles of $n^{1-o(1)}$ different \nolinebreak lengths.
\end{thm}


Our proof of Theorem \ref{thm:main} is constructive: it gives a polynomial-time algorithm for finding cycles of $n^{1-o(1)}$ different lengths in a Hamiltonian graph of minimum degree $3$, provided a Hamilton cycle is specified. 

\section{A sketch and main ideas}\label{subsec:sketch}

The most general overarching idea that we employ is to split the Hamilton cycle into pieces (usually paths or pairs of paths) and then find paths with lengths on a different ``scale'' in different parts. 
To illustrate what we mean, let us consider the following situation. Suppose that we managed to split our Hamilton cycle into two paths $P_1,P_2$, such that there are still many chords inside the vertex-set of each $P_i$ (or, more precisely, that inside each $P_i$ there is a linear number of vertices touching a chord whose other endpoint is also on $P_i$). Suppose that we found $k = \Omega(\sqrt{n})$ paths $Q_1,\dots,Q_k$ between the endpoints of $P_1$ (which only use the vertices of $P_1$), such that $|Q_1|,\dots,|Q_k|$ are all different and all belong to an interval of width $\sqrt{n}$. Suppose further that we found $\ell = \Omega(\sqrt{n})$ paths $R_1,\dots,R_{\ell}$ between the endpoints of $P_2$ (which only use the vertices of $P_2$), such that the lengths of any two of these paths are at least $\sqrt{n}$ apart, namely, $||R_i| - |R_j|| > \sqrt{n}$ for all $i \neq j$. In this situation, we can combine any one of the $Q_i$'s with any one the $R_j$'s, joining them into a cycle of length $|Q_i| + |R_j|$.
The crucial point is that the $k\ell$ numbers $|Q_i| + |R_j|$ are all different. In other words, we use the ``condensed'' lengths $Q_1,\dots,Q_k$ to ``fill in the gaps'' between the ``spread-out'' lengths $R_1,\dots,R_{\ell}$ (see \Cref{lem:sum-sets} for the details on this). In total, this would give us $k\ell = \Omega(n)$ different cycle lengths. Hence, achieving both above goals would establish Conjecture \ref{conj:main}.




We believe that both above statements should be true, namely, that one can find both $\Omega(\sqrt{n})$ distinct path lengths all contained in an interval of width $\sqrt{n}$ and $\Omega(\sqrt{n})$ path lengths which are $\sqrt{n}$ apart. Observe that both of these statements are essentially implied by \Cref{conj:main}, and that our main result shows that both hold asymptotically 
(i.e., with $\sqrt{n}$ replaced by $n^{1/2-o(1)}$).
On the other hand, these statements shift the difficulty from finding many lengths (note that there have been a number of proofs that there are at least $\sqrt{n}$ different lengths over the years) to controlling what kind of lengths \nolinebreak we \nolinebreak find. 

Our actual strategy for tackling Conjecture \ref{conj:main} is a bit more involved. Instead of splitting our cycle into just two parts, we split it into a larger number $k$ of parts (with $k$ to be chosen as roughly $\sqrt{\log n}$). Here each part will be a pair of cycle sections (subpaths of the cycle) with at least $n^{1-o(1)}$ chords between them, with different section-pairs situated ``on top of'' each other (see \Cref{fig:0.2}).
Now, with the goal of finding $n^{1-\eps}$ different lengths (where $\varepsilon$ is an appropriately chosen vanishing function of $n$), we shall proceed as follows. Inside the first of the $k$ parts, we shall find $\Omega(n^{\eps})$ path lengths all belonging to an interval of width $n^{\eps}$.
Then, inside the second part, we shall find about $\Omega(n^{\eps})$ lengths $\ell_1 < \ldots < \ell_t$ such that any two consecutive lengths are $\Theta(n^{\eps})$ apart, namely $\ell_{i+1} - \ell_i = \Theta(n^{\eps})$ for all $i$. Now, by combining the paths we found in these two parts, we will get $\Omega(n^{2\eps})$ different path lengths, all belonging to an interval of width $O(n^{2\eps})$, and only using vertices from the first two parts of the partition. Continuing in this manner, we will find inside the third part $\Omega(n^{\eps})$ lengths which are $\Theta(n^{2\eps})$ apart, inside the fourth part $\Omega(n^{\eps})$ lengths which are $\Theta(n^{3\eps})$ apart, and so on. This will always allow us to combine the new lengths we find with the lengths found so far to get $\Omega(n^{i\eps})$ different path lengths, all belonging to an interval of width $O(n^{i\eps})$, only using vertices from the first $i$ parts.
In each iteration we will actually lose a polylogarithmic factor in the number of paths we find, which will result in the optimal number of iterations being $\sqrt{\frac{\log n}{\log \log n}}$ (this corresponds to having $\varepsilon = \sqrt{\frac{\log \log n}{\log n}}$). After this number of iterations, we will find $n^{1-o(1)}$ different lengths. This iterative process is handled in \Cref{lem:main-part}.

Let us now focus on a single iteration and sketch the main ideas involved. For simplicity, suppose that this is the third iteration, namely, that our goal is to find (inside the third of the $k$ parts of the partition) $\Omega(n^{\eps})$ path lengths $\ell_1 < \ldots <\ell_t$ with $\ell_{i+1} - \ell_i = \Theta(n^{2\eps})$ for all $i$. Up to this step, we have already found $\Theta(n^{2\eps})$ path lengths in an interval of width $O(n^{2\eps})$ inside the first two parts. Now, we consider a maximum collection $e_1,\dots,e_m$ of chords inside the third part, such that for all $i \neq j$, the lengths of $e_i$ and $e_j$ differ by at least $n^{2\varepsilon}$. Each chord $e_i$ gives rise to a path inside the third part (namely, the path that consists of the chord and pieces of the cycle), and the lengths of any two of these $m$ paths differ by at least $n^{2\varepsilon}$.   
Now, observe that if $m \geq n^{1-3\varepsilon}$, then by combining these paths with the $\Theta(n^{2\eps})$ path lengths we found in the first two parts of the partition, we obtain altogether $m \cdot \Omega(n^{2\eps}) = \Omega(n^{1-\varepsilon})$ different cycle lengths, and thus achieve our goal already at this stage. So we may assume that $m \leq n^{1-3\varepsilon}$. 
Since $e_1,\dots,e_m$ is a maximal family, the length of any other chord must be at distance at most $n^{2\varepsilon}$ to that of one of the $e_i$'s. 
By averaging (and as each part of the partition contains $n^{1-o(1)}$ chords), we see that there is a family $E$ of at least $n^{1-o(1)}/m \geq n^{3\varepsilon - o(1)}$ different chords, whose lengths all belong to an interval of width $n^{2\varepsilon}$. The reason such a family $E$ is useful is as follows: Suppose we partition the two cycle sections of the third part into subpaths $X_1,X_2,\ldots$ of length $n^{2\varepsilon}$. Then, for any two such subpaths $X_i,X_j$ on the left side, which are not consecutive (and hence are at distance larger than $n^{2\varepsilon}$ on the path), any chord touching $X_1$ must interlace (i.e. cross) any chord touching $X_2$. For if not, then the difference of the lengths of these two chords is larger than $n^{2\varepsilon}$, contradicting the fact that both lengths belong to an interval of width $n^{2\varepsilon}$. So we see that $E$ decomposes into pairwise-interlacing pieces, see \Cref{fig:0.3} for an illustration. This structure, together with some additional arguments, then allows us to find the 
$\Omega(n^{\eps})$ desired path lengths $\ell_1 < \ldots <\ell_t$. We remark that while it is not hard to find such lengths with $\ell_{i+1}-\ell_i= \Omega(n^{2\eps})$, which already allows us to find $\Omega(n^{3\eps})$ lengths, it is essential for the next iteration that these lengths are not too far apart, in other words ensuring in addition that $\ell_{i+1}-\ell_i \le O(n^{2\eps})$ is crucial in order to be able to continue our argument.

Executing the above strategy presents significant technical difficulties, in part due to the need to join the various pieces in suitable ways. This sometimes makes various parts of the argument a bit cumbersome, and we try to alleviate this situation with a number of figures illustrating the argument. 

Finally, we note that even splitting the cycle into pieces (i.e. subpaths or pairs of sections), a preliminary step for the above strategy, is non-trivial and requires some work. To illustrate this, note that if our graph only consisted of ``diameter'' chords (joining vertices which are as far apart on the cycle as possible), then it is not possible to split the cycle into two paths with each of them still containing many chords. The way we go around this issue is to show that one can ``reroute'' the original Hamilton cycle using two chords (and large parts of the Hamilton cycle), thus obtaining a new cycle which admits the desired split.

\section{Setting up the stage}

Let us introduce some terminology that we will use in the proof. 
For a path $P$, denote by $|P|$ the number of edges in $P$. Given a path or a cycle, we will call its subpaths \textit{sections}.
We will often consider chords between two sections of a cycle, and hence it is convenient to have the following setup. A {\em section-pair} is a pair of vertex-disjoint paths $X,Y$. 
We will always denote the endpoints of $X$ by 
$x^{\text{t}}$ and $x^{\text{b}}$, where $x^{\text{t}}$ is called the top and $x^{\text{b}}$ the bottom. Similarly, the endpoints of $Y$ are
denoted $y^{\text{t}},y^{\text{b}}$. For distinct $x_1,x_2 \in X$, we say that $x_1$ is {\em above} $x_2$ if $x_1$ is closer to $x^{\text{t}}$ along $X$ than $x_2$; otherwise we say that $x_1$ is {\em below} $x_2$. For sets $X_1,X_2 \subseteq X$, we say that $X_1$ is {\em above} (resp. {\em below}) $X_2$ if every $x_1 \in X_1$ is above (resp. below) every $x_2 \in X_2$. When considering a sequence $X_1,\dots,X_t$ of disjoint subsets of $X$, unless otherwise specified, we assume they are labeled in such a way that $X_i$ is below $X_j$ for all $1 \leq i < j \leq t$.  
We use $x_i^{\text{t}}$ (resp. $x_i^{\text{b}}$) to denote the top (resp. bottom) vertex of $X_i$.

For $x_1,x_2 \in X$, denote by $X[x_1,x_2]$ the subpath of $X$ between $x_1$ and $x_2$, and by $d_X(x_1,x_2) = |X[x_1,x_2]|$ the length of this subpath, namely, the distance between $x_1$ and $x_2$ along $X$.
Define $Y[y_1,y_2]$ and $d_Y(y_1,y_2)$ analogously for $y_1,y_2 \in Y$.  A \textit{subsection pair} of a section pair $X,Y$ is a section pair consisting of a subpath of $X$ and a subpath of $Y$.

A chord is an edge with one endpoint in $X$ and one in $Y$. Let $(x_1,y_1),(x_2,y_2)$ be two chords which have no common vertices. We say that $(x_1,y_1),(x_2,y_2)$
are {\em parallel} if $x_i$ is above $x_{3-i}$ and $y_i$ is above $y_{3-i}$ for some $i = 1,2$; otherwise we say that $(x_1,y_1),(x_2,y_2)$ are {\em interlacing}. In other words, $(x_1,y_1),(x_2,y_2)$ are interlacing if $x_1$ is above $x_2$ but $y_1$ is below $y_2$, or vice versa.

The following statement is equivalent to (the symmetric case of) the Erd\"os-Szekeres lemma. 
\begin{lem}\label{lem:ES}
Let $k \geq 1$, and let $X,Y$ be a section-pair with at least $(k-1)^2 + 1$ chords, no two of which share vertices. Then there is a set $E$ of $k$ chords such that either every two chords in $E$ are parallel or every two chords in $E$ are interlacing.  
\end{lem}

Given a collection of disjoint subsection pairs $X_1,Y_1;\ldots; X_t,Y_t$ (where as usual we will assume that $X_i$ is below $X_j$ for all $i<j$), we say that the collection is parallel if also $Y_i$ is below $Y_j$ for all $i<j$, and we say that it is interlacing if $Y_i$ is above $Y_j$ for all $i<j$. See Figures 2 and 3.

\begin{figure}
\RawFloats
\begin{minipage}[t]{0.32\textwidth}
\centering
\captionsetup{width=\textwidth}
\begin{tikzpicture}[scale=0.75]
\defPt{0}{0}{x_b}
\defPt{0}{4}{x_t}
\defPt{3}{0}{y_b}
\defPt{3}{4}{y_t}
\defPt{-0.5}{2.75}{x}
\defPt{3.5}{1.25}{y}

\draw[line width= 1.25pt] (x_t) to[out=-120,in=120]  (x_b);
\draw[line width= 1.25pt] (y_t) to[out=-60,in=60]  (y_b);

\draw[line width= 1.75pt,red] (x_t) to[out=-120,in=75]  (x);
\draw[line width= 1.75pt,red] (y_t) to[out=-60,in=80]  (y);
\draw[red,line width= 1.75 pt] (x) -- (y);

\draw[line width= 1pt, dotted] (x_t) to[out=45,in=135]  (y_t);
\draw[line width= 1pt, dotted] (x_b) to[out=-45,in=-135]  (y_b);

\node[] at ($(x_t)+(-0.4,0.4)$) {\small $x^{\text{t}}$};
\node[] at ($(x_b)+(-0.4,-0.4)$) {\small $x^{\text{b}}$};
\node[] at ($(y_t)+(0.4,0.4)$) {\small $y^{\text{t}}$};
\node[] at ($(y_b)+(0.4,-0.4)$) {\small $y^{\text{b}}$};

\node[] at ($(x)+(-0.4,0)$) {\small $x$};
\node[] at ($(y)+(0.4,0)$) {\small $y$};

\draw[] (x_b) \smvx;
\draw[] (x_t) \smvx;
\draw[] (y_b) \smvx;
\draw[] (y_t) \smvx;
\draw[] (x) \smvx;
\draw[] (y) \smvx;

\end{tikzpicture}
\caption{A section pair with a chord $(x,y)$ and its corresponding trivial path marked in red.}
\label{fig:0.1}
\end{minipage}\hfill
\begin{minipage}[t]{0.32\textwidth}
\centering
\captionsetup{width=\textwidth}
\begin{tikzpicture}[scale=0.8]

\defPt{0}{0}{x_b}
\defPt{0}{4.5}{x_t}
\defPtm{($0.075*(x_b)+0.925*(x_t)$)}{x6}
\defPtm{($0.225*(x_b)+0.775*(x_t)$)}{x5}

\defPtm{($0.425*(x_b)+0.575*(x_t)$)}{x4}
\defPtm{($0.575*(x_b)+0.425*(x_t)$)}{x3}

\defPtm{($0.775*(x_b)+0.225*(x_t)$)}{x2}
\defPtm{($0.925*(x_b)+0.075*(x_t)$)}{x1}

\defPt{3.5}{0}{y_b}
\defPt{3.5}{4.5}{y_t}

\defPtm{($0.075*(y_b)+0.925*(y_t)$)}{y6}
\defPtm{($0.225*(y_b)+0.775*(y_t)$)}{y5}

\defPtm{($0.425*(y_b)+0.575*(y_t)$)}{y4}
\defPtm{($0.575*(y_b)+0.425*(y_t)$)}{y3}

\defPtm{($0.775*(y_b)+0.225*(y_t)$)}{y2}
\defPtm{($0.925*(y_b)+0.075*(y_t)$)}{y1}

\draw[color=white, fill=black!50!,fill opacity=0.3] (x1) -- (x2) -- (y2) -- (y1) -- cycle;
\draw[color=white, fill=black!50!,fill opacity=0.3] (x3) -- (x4) -- (y4) -- (y3) -- cycle;
\draw[color=white, fill=black!50!,fill opacity=0.3] (x5) -- (x6) -- (y6) -- (y5) -- cycle;

\draw[] (x_b) \smvx;
\draw[] (x_t) \smvx;
\foreach \i in {1,...,6}
{
\draw[] (x\i) \smvx;
\draw[] (y\i) \smvx;
}
\draw[] (y_b) \smvx;
\draw[] (y_t) \smvx;

\draw[line width= 0.75 pt] (x_b) -- (x_t);
\draw[line width= 0.75 pt] (y_b) -- (y_t);

\draw[line width= 2 pt] (x1) -- (x2);
\draw[line width= 2 pt] (x3) -- (x4);
\draw[line width= 2 pt] (x5) -- (x6);

\draw[line width= 2 pt] (y1) -- (y2);
\draw[line width= 2 pt] (y3) -- (y4);
\draw[line width= 2 pt] (y5) -- (y6);

\foreach \i in {1,...,6}
{
\draw[line width= 1 pt, dashed] (x\i) -- (y\i);
}

\draw[line width= 1 pt] ($0.25*(x1)+0.75*(x2)$) -- ($0.5*(y1)+0.5*(y2)$);
\draw[line width= 1 pt] ($0.5*(x1)+0.5*(x2)$) -- ($0.75*(y1)+0.25*(y2)$);
\draw[line width= 1 pt] ($0.75*(x1)+0.25*(x2)$) -- ($0.25*(y1)+0.75*(y2)$);

\draw[line width= 1 pt] ($0.25*(x3)+0.75*(x4)$) -- ($0.5*(y3)+0.5*(y4)$);
\draw[line width= 1 pt] ($0.5*(x3)+0.5*(x4)$) -- ($0.75*(y3)+0.25*(y4)$);
\draw[line width= 1 pt] ($0.75*(x3)+0.25*(x4)$) -- ($0.25*(y3)+0.75*(y4)$);

\draw[line width= 1 pt] ($0.25*(x5)+0.75*(x6)$) -- ($0.5*(y5)+0.5*(y6)$);
\draw[line width= 1 pt] ($0.5*(x5)+0.5*(x6)$) -- ($0.75*(y5)+0.25*(y6)$);
\draw[line width= 1 pt] ($0.75*(x5)+0.25*(x6)$) -- ($0.25*(y5)+0.75*(y6)$);

\node[] at ($0.5*(x1)+0.5*(x2)+(-0.5,0)$) {\small $X_1$};
\node[] at ($0.5*(x3)+0.5*(x4)+(-0.5,0)$) {\small $X_2$};
\node[] at ($0.5*(x5)+0.5*(x6)+(-0.5,0)$) {\small $X_3$};

\node[] at ($0.5*(y1)+0.5*(y2)+(0.5,0)$) {\small $Y_1$};
\node[] at ($0.5*(y3)+0.5*(y4)+(0.5,0)$) {\small $Y_2$};
\node[] at ($0.5*(y5)+0.5*(y6)+(0.5,0)$) {\small $Y_3$};

\node[] at ($(x_t)+(-0.5,0.1)$) {\small $x^{\text{t}}$};
\node[] at ($(x_b)+(-0.5,0.1)$) {\small $x^{\text{b}}$};
\node[] at ($(y_t)+(0.5,0.1)$) {\small $y^{\text{t}}$};
\node[] at ($(y_b)+(0.5,0.1)$) {\small $y^{\text{b}}$};


\end{tikzpicture}
\caption{A parallel collection of subsection pairs}
\label{fig:0.2}
\end{minipage}\hfill
\begin{minipage}[t]{0.32\textwidth}
\centering
\captionsetup{width=\textwidth}
\begin{tikzpicture}[scale=0.8]

\defPt{0}{0}{x_b}
\defPt{0}{4.5}{x_t}
\defPtm{($0.075*(x_b)+0.925*(x_t)$)}{x6}
\defPtm{($0.225*(x_b)+0.775*(x_t)$)}{x5}

\defPtm{($0.425*(x_b)+0.575*(x_t)$)}{x4}
\defPtm{($0.575*(x_b)+0.425*(x_t)$)}{x3}

\defPtm{($0.775*(x_b)+0.225*(x_t)$)}{x2}
\defPtm{($0.925*(x_b)+0.075*(x_t)$)}{x1}

\defPt{3.5}{0}{y_b}
\defPt{3.5}{4.5}{y_t}

\defPtm{($0.075*(y_b)+0.925*(y_t)$)}{y2}
\defPtm{($0.225*(y_b)+0.775*(y_t)$)}{y1}

\defPtm{($0.425*(y_b)+0.575*(y_t)$)}{y4}
\defPtm{($0.575*(y_b)+0.425*(y_t)$)}{y3}

\defPtm{($0.775*(y_b)+0.225*(y_t)$)}{y6}
\defPtm{($0.925*(y_b)+0.075*(y_t)$)}{y5}

\draw[color=white, fill=black!50!,fill opacity=0.3] (x1) -- (x2) -- (y2) -- (y1) -- cycle;
\draw[color=white, fill=black!50!,fill opacity=0.3] (x3) -- (x4) -- (y4) -- (y3) -- cycle;
\draw[color=white, fill=black!50!,fill opacity=0.3] (x5) -- (x6) -- (y6) -- (y5) -- cycle;

\draw[] (x_b) \smvx;
\draw[] (x_t) \smvx;
\foreach \i in {1,...,6}
{
\draw[] (x\i) \smvx;
\draw[] (y\i) \smvx;
}
\draw[] (y_b) \smvx;
\draw[] (y_t) \smvx;

\draw[line width= 0.75 pt] (x_b) -- (x_t);
\draw[line width= 0.75 pt] (y_b) -- (y_t);

\draw[line width= 2 pt] (x1) -- (x2);
\draw[line width= 2 pt] (x3) -- (x4);
\draw[line width= 2 pt] (x5) -- (x6);

\draw[line width= 2 pt] (y1) -- (y2);
\draw[line width= 2 pt] (y3) -- (y4);
\draw[line width= 2 pt] (y5) -- (y6);

\foreach \i in {1,...,6}
{
\draw[line width= 1 pt, dashed] (x\i) -- (y\i);
}

\draw[line width= 1 pt] ($0.25*(x1)+0.75*(x2)$) -- ($0.5*(y1)+0.5*(y2)$);
\draw[line width= 1 pt] ($0.5*(x1)+0.5*(x2)$) -- ($0.75*(y1)+0.25*(y2)$);
\draw[line width= 1 pt] ($0.75*(x1)+0.25*(x2)$) -- ($0.25*(y1)+0.75*(y2)$);

\draw[line width= 1 pt] ($0.25*(x3)+0.75*(x4)$) -- ($0.5*(y3)+0.5*(y4)$);
\draw[line width= 1 pt] ($0.5*(x3)+0.5*(x4)$) -- ($0.75*(y3)+0.25*(y4)$);
\draw[line width= 1 pt] ($0.75*(x3)+0.25*(x4)$) -- ($0.25*(y3)+0.75*(y4)$);

\draw[line width= 1 pt] ($0.25*(x5)+0.75*(x6)$) -- ($0.5*(y5)+0.5*(y6)$);
\draw[line width= 1 pt] ($0.5*(x5)+0.5*(x6)$) -- ($0.75*(y5)+0.25*(y6)$);
\draw[line width= 1 pt] ($0.75*(x5)+0.25*(x6)$) -- ($0.25*(y5)+0.75*(y6)$);

\node[] at ($0.5*(x1)+0.5*(x2)+(-0.5,0)$) {\small $X_1$};
\node[] at ($0.5*(x3)+0.5*(x4)+(-0.5,0)$) {\small $X_2$};
\node[] at ($0.5*(x5)+0.5*(x6)+(-0.5,0)$) {\small $X_3$};

\node[] at ($0.5*(y1)+0.5*(y2)+(0.5,0)$) {\small $Y_1$};
\node[] at ($0.5*(y3)+0.5*(y4)+(0.5,0)$) {\small $Y_2$};
\node[] at ($0.5*(y5)+0.5*(y6)+(0.5,0)$) {\small $Y_3$};

\node[] at ($(x_t)+(-0.5,0.1)$) {\small $x^{\text{t}}$};
\node[] at ($(x_b)+(-0.5,0.1)$) {\small $x^{\text{b}}$};
\node[] at ($(y_t)+(0.5,0.1)$) {\small $y^{\text{t}}$};
\node[] at ($(y_b)+(0.5,0.1)$) {\small $y^{\text{b}}$};



\end{tikzpicture}
\caption{An interlacing collection of subsection pairs}
\label{fig:0.3}
\end{minipage}
\end{figure}
\vspace{-0.2cm}

The {\em length} of a chord $(x,y)$ is defined as $d_X(x^{\text{t}},x) + d_Y(y^{\text{t}},y)$, which is one less than the length of the path $X[x^{\text{t}},x],(x,y),Y[y,y^{\text{t}}]$. This path will be called the {\em trivial path} corresponding to the chord $(x,y)$.

A {\em top-bottom path-pair} (for the section-pair $X,Y$) is a pair of vertex-disjoint paths which are contained in $X \cup Y$, start at $x^{\text{t}},y^{\text{t}}$, and end at $x^{\text{b}},y^{\text{b}}$. So, for example, $X,Y$ is a (trivial) top-bottom path-pair. The length of a top-bottom path-pair $\mathcal{P} = (P_1,P_2)$ is defined as $|P_1| + |P_2|$ and denoted by $|\mathcal{P}|$.  

Given $x \in X, y \in Y$, a path from $x$ to $y$ is called a \textit{below-path} if it only uses vertices of $X$ and $Y$ which are below $x$ and $y$, respectively.

In several points in the proof, we will split our cycle into disjoint parts (paths or section pairs) and concatenate paths (or path-pairs) which we find in different parts. Evidently, the length of the concatenated path is the sum of lengths of the individual paths. Hence, the set of path lengths we obtain in this way is the {\em sum-set} of the sets of path lengths we find in each of the different parts.
Formally, given a sequence of sets $L_1,\ldots, L_t$, we define the sum-set $L_1+\ldots+L_t:=\{\ell_1+\ldots+\ell_t \mid \ell_i \in L_i\}$. Given a set $L$ of integers, we say that two elements of $L$ are {\em consecutive} (with respect to $L$) if there is no element of $L$ between them.
By {\em interval} we always mean an interval of natural numbers. 
We will often use the trivial fact that if $x,y$ belong to an interval $I$ then $|x - y| \leq |I|-1$.
For numbers $x,y$, we say that $x,y$ are {\em at least $b$ apart} (or just {\em $b$ apart}) if $|x - y| \geq b$, and {\em at most $b$ apart} if $|x - y| \leq b$.  
The following auxiliary lemmas about sum-sets will come in handy later. 
\begin{lem}\label{lem:sum-sets}
Let $L_1,\ldots, L_t \subseteq \mathbb{N}$ with the property that for all $1\le i\le t$, every two consecutive elements of $L_i$ are at least $a$ and at most $b$ apart. Then there is a subset of $L_1+\ldots + L_t$ of size $1+\sum_{i=1}^t (|L_i|-1)$ in which every two consecutive elements are at least $a$ and at most $b$ apart.  
\end{lem}
\begin{proof}
 We will prove the lemma by induction on $t$. If $t=1$ then the statement is trivial. Assume it holds for $t-1$ so that we can find a subset $S$ of $L_1+\ldots + L_{t-1}$ of size $1+\sum_{i=1}^{t-1} (|L_i|-1)$ with all consecutive elements being at least $a$ and at most $b$ apart. 
 Let $x_1 < \ldots < x_s$ be the elements of $S$, and let $y_1 < \ldots < y_r$ be the elements of $L_t$. Then $\{ x_1 + y_1 < x_2 + y_1 < \ldots < x_s + y_1 < x_s + y_2 < \ldots < x_s + y_r\}$ is a subset of $L_1 + \ldots + L_t$ of size $|S| + |L_t| - 1 = 1+\sum_{i=1}^t (|L_i|-1)$ having the desired property.
\end{proof}

\begin{lem}\label{lem:spread-close}
Let $L_1,L_2$ be sets of positive integers. Suppose that each $L_i$ is a subset of an interval of size $\ell_i$ and that any two elements of $L_1$ are at least $\ell_2$ apart. Then $L_1+L_2$ has $|L_1||L_2|$ elements, all belonging into an interval of size $\ell_1+\ell_2$.
\end{lem}
\begin{proof}
Let $r_1< \ldots < r_t$ be the elements of $L_1$ and $s_1 < \ldots < s_p$ be the elements of $L_2$. We claim that all $r_i+s_j$ are distinct. To see this, assume that $r_i+s_j=r_{i'}+s_{j'}$. If $i\neq i'$ then we must have $|r_i-r_{i'}| \ge \ell_2$. However, $|s_j-s_j'| <\ell_2$, and so this is impossible. Hence $i=i'$. This further implies that $j=j'$ and completes the proof. This shows that $L_1+L_2$ indeed has $|L_1||L_2|$ elements. Furthermore, notice that all elements in $L_1+L_2$ belong into $[r_1+s_1,r_t+s_p]$ and since $r_t-r_1 < \ell_1$ and $s_p-s_1 < \ell_2$, this interval indeed has size at most $\ell_1+\ell_2$.
\end{proof}

All our logarithms are in base $2$ unless otherwise specified. We will omit floor/ceiling signs whenever these are not crucial. 

\section{Finding cycles of many different lengths}

In this section we will prove our main result, \Cref{thm:main}. We begin, in the following subsection, by proving a number of lemmas which provide us with the main tools to attack the problem. We then proceed to put everything together in the subsequent subsection.

\subsection{Main lemmas}

The following lemma states that in every section-pair with many chords, one can find either a parallel or an interlacing collection of subsection pairs, each of which still has many chords. 

\begin{lem}\label{lem:initial_split}
Let $X,Y$ be a section-pair with $m$ chords, and let $k \geq 1$. Suppose that no vertex of $X \cup Y$ is incident to more than $\frac{m}{10k^2}$ chords. 
Then there is either a parallel or an interlacing collection of subsection pairs $X_1,Y_1;\dots;X_k,Y_k$ such that $e(X_i,Y_i) \geq \Omega(\frac{m}{k^4})$ for all $1 \leq i \leq k$.
\end{lem}
\begin{proof}
We partition $X$ into subpaths $X'_1,\dots,X'_{s},X'_{s+1}$ such that for each $1 \leq i \leq s$, $X'_i$ touches at least $\frac{m}{10k^2}$ and at most $\frac{m}{5k^2}$ chords, and such that $X'_{s+1}$ touches less than $\frac{m}{10k^2}$ chords. We do this as follows: Take $X'_1$ to be the minimum initial segment of $X$ which touches at least $\frac{m}{10k^2}$ chords. By minimality, the segment obtained by removing the last vertex $x$ of $X'_1$ touches less than $\frac{m}{10k^2}$ chords. Also, $x$ itself touches at most $\frac{m}{10k^2}$ by the assumption of the lemma. So overall, $X'_1$ touches at most $\frac{m}{5k^2}$ chords. We now remove $X'_1$ and continue in this fashion. As long as the remaining subpath touches at least $\frac{m}{10k^2}$ chords, we can extract an initial segment of it which touches at least $\frac{m}{10k^2}$ and at most $\frac{m}{5k^2}$ chords. In the end, we are left with a subpath touching less than $\frac{m}{10k^2}$ chords, and we take this subpath to be $X'_{s+1}$. 

In the same way, partition $Y$ into subpaths $Y'_1,\dots,Y'_t,Y'_{t+1}$ with each $Y'_i$, $1 \leq i \leq t$, touching at least $\frac{m}{10k^2}$ and at most $\frac{m}{5k^2}$ chords, and with $Y'_{t+1}$ touching less than $\frac{m}{10k^2}$ chords. 
Clearly, 
$4k^2 \leq (m - \frac{m}{10k^2})/\frac{m}{5k^2} \leq s,t \leq 10k^2$. 
Define an auxiliary bipartite graph $G$ with sides $[s]$ and $[t+1]$, in which $(i,j)$ is an edge if $e(X'_i,Y'_j) \geq \frac{m}{400k^4}$. 
Let $I \subseteq [s]$. We claim that $|N_G(I)| \geq |I|/4$. To see this, note that
\begin{align*}
|I| \cdot \frac{m}{10k^2} &\leq \sum_{i \in I}{e(X'_i,Y)} \leq 
\sum_{j \in N_G(I)}{e(X,Y'_j)} + (t+1) \cdot |I| \cdot \frac{m}{400k^4} \\ &\leq 
|N_G(I)| \cdot \frac{m}{5k^2} + 2t \cdot |I| \cdot \frac{m}{400k^4} \leq 
|N_G(I)| \cdot \frac{m}{5k^2} + |I| \cdot \frac{m}{20k^2}.
\end{align*}
Rearranging gives $|N_G(I)| \geq |I|/4$. Now, by a well-known generalization of Hall's theorem, there is a matching in $G$ which saturates at least $s/4 \geq k^2$ of the elements of $[s]$. 
By Lemma \ref{lem:ES}, such a matching contains either a parallel or an interlacing family of chords of size $k$. It is easy to see that such a family gives sets $X_1,\dots,X_k,Y_1,\dots,Y_k$ as in the statement of the lemma.
\end{proof}

In a section-pair $X,Y$ with many chords, we will be able to find many paths of different lengths, for example, between $x^t$ and $y^b$. However, in order to be able to join these paths into cycles of many different lengths, we need to find top-bottom path-pairs that include these paths. This is necessary in order to combine these lengths with lengths found in other section-pairs above or below $X,Y$.  
In order to find such top-bottom path-pairs,
it is useful to have a special chord $(x,y) \in E(X,Y)$ which interlaces all (or many) other chords in $E(X,Y)$. Indeed, observe that we can use such a chord $(x,y)$ to walk from $y^t$ to $x^b$, while using the chords which interlace $(x,y)$ to obtain many path lengths from $x^t$ to $y^b$. This way, the two paths (one from $x^t$ to $y^b$ and one from $y^t$ to $x^b$) do not intersect each other, see the left part of Figure 10 (where the special chord $(x,y)$ is denoted by $(x_j,y_j)$).  

Evidently, a chord $(x,y)$ as above does not exist if all chords in $E(X,Y)$ are pairwise-parallel. Hence, it is reasonable to try and partition $X,Y$ into parallel subsections and try and find such a chord $(x,y)$ in each of them. This again may fail if in a given part there is a vertex of a very high degree. As the following lemma shows, this is essentially the only obstacle.  
For an illustration of the two outcomes of Lemma \ref{lem:splitting_process}, we refer the reader to Figures \ref{fig:split-case-1} and \ref{fig:split-case-2}. 
As the statement of Lemma \ref{lem:splitting_process} suggests, we cannot guarantee that the special chord $(x,y)$ interlaces a constant fraction of all other chords; we have to pay a log factor. 

\begin{lem}\label{lem:splitting_process}
Let $X,Y$ be a section-pair with $m\ge 2$ chords. Then there is a parallel collection of subsection pairs $X_1,Y_1;\dots;X_t,Y_t$ such that $\sum_{i=1}^t {e(X_i,Y_i)} \geq m/24$ and one of the following holds:
\begin{enumerate}
    \item For every $i$, there is a vertex in $X_i \cup Y_i$ which is incident to at least $\frac{e(X_i,Y_i)}{6\log m}$ chords in $E(X_i,Y_i)$.
    \item For every $i$, there is a chord in $E(X_i,Y_i)$ which interlaces at least $\frac{e(X_i,Y_i)}{6\log m}$ of the chords in $E(X_i,Y_i)$. 
\end{enumerate}
\end{lem}
\begin{proof}
We will produce a partition $X_1,\ldots,X_{t'}$ of $X$ and a partition $Y_1,\ldots, Y_{t'}$ of $Y$ such that \linebreak $\sum_{i=1}^{t'} {e(X_i,Y_i)} \geq m/12,$ and each pair $X_i,Y_i$ satisfies the condition of one of the items 1 or 2. 
Evidently, for one of the items, the pairs satisfying it would contribute at least half of all chords. Keeping only the pairs satisfying this item, we will obtain the desired collection.

We now construct such partitions $X_1,\ldots,X_{t'}$ and $Y_1,\ldots, Y_{t'}$ through the following process. We start with the trivial partition $\{X\}$ of $X$ and $\{Y\}$ of $Y$. Note that we may assume that $m\ge 16$, or the trivial partition already satisfies the condition of item 1. Suppose that at a given step of the process, there is a pair of sets $X_i,Y_i$ which violates the conditions of both items 1 and 2.
Partition $X_i$ into two subsections $X',X''$ with $X'$ above $X''$, such that $X'$ and $X''$ each touch at least $\left( \frac{1}{2} - \frac{1}{6\log m} \right)e(X_i,Y_i)$ and at most $\left( \frac{1}{2} + \frac{1}{6\log m} \right)e(X_i,Y_i)$ of the chords between $X_i$ and $Y_i$. 
This is possible since every vertex of $X_i$ is incident to less than $\frac{e(X_i,Y_i)}{6\log m}$ of these chords (by the assumption that item 1 fails). 
Let $x$ be the lowest vertex of $X'$ which is incident to a chord, and let $y$ be the lowest vertex of $Y_i$ which is adjacent to $x$. 
Partition $Y$ into subsections $Y',Y''$, such that $Y'$ is above $Y''$ and $y$ is the lowest vertex of $Y'$. Observe that if a chord $e \in E(X',Y'') \cup E(X'',Y')$ does not contain $y$, then $e$ interlaces $(x,y)$. So the number of such chords is less than $\frac{e(X_i,Y_i)}{6\log m}$.
Also, $e(X_i,y) < \frac{e(X_i,Y_i)}{6\log m}$. (Here we used the assumption that $X_i,Y_i$ violates both items 1 and 2). 
We conclude that
$e(X',Y'') + e(X'',Y') < \frac{e(X_i,Y_i)}{3\log m}$.
It follows that $e(X',Y') = e(X',Y_i) - e(X',Y'') \geq \left( \frac{1}{2} - \frac{1}{2\log m} \right)e(X_i,Y_i)$, and similarly $e(X'',Y'') \geq \left( \frac{1}{2} - \frac{1}{2\log m} \right)e(X_i,Y_i)$. Note that we also have the upper bound $e(X',Y'),e(X'',Y'') \leq \left( \frac{1}{2} + \frac{1}{6\log m} \right)e(X_i,Y_i)$. 
We now define new partitions by replacing $X_i$ with $X',X''$ and $Y_i$ with $Y',Y''$. This way, we may continue the process until every pair $X_i,Y_i$ satisfies the condition of one of the items 1 or 2. Indeed, at each step, we replace some pair $X_i,Y_i$ with pairs having less edges than $X_i,Y_i$. If the number of edges of some pair is at most $6$, then this pair trivially satisfies item 1. Hence, the process must terminate.

To each pair $(X_i,Y_i)$ appearing in the course of the process, we assign a binary string as follows. Assign the initial pair $(X,Y)$ the empty string. If $(X',Y'),(X'',Y'')$ are obtained by splitting $(X_i,Y_i)$ as above, and $(X_i,Y_i)$ is assigned the string $\sigma$, then let $(X',Y')$ be assigned the string $\sigma,0$ and $(X'',Y'')$ the string $\sigma,1$. Let $h(X_i,Y_i)$ be the length of the string assigned to a pair $X_i,Y_i$. It is easy to show, by induction, \nolinebreak that  
\begin{equation}\label{eq:splitting}
\sum\left( \frac{1}{2} \right)^{h(X_i,Y_i)} = 1
\end{equation}
is preserved throughout the process, where the sum ranges over all pairs $X_i,Y_i$ at a given moment. Moreover, by our upper and lower bounds on $e(X',Y'),e(X'',Y'')$, we get by induction that
$$
0<m \cdot \left( \frac{1}{2} - \frac{1}{2\log m} \right)^{h(X_i,Y_i)} \leq e(X_i,Y_i) \leq m \cdot \left( \frac{1}{2} + \frac{1}{6\log m} \right)^{h(X_i,Y_i)}
$$
for every pair $(X_i,Y_i)$ appearing in the process. In particular, $h(X_i,Y_i) \le 2\log m$ for each such pair $(X_i,Y_i)$. Hence $$e(X_i,Y_i) \geq
m \cdot 
\left( \frac{1}{2} - \frac{1}{2\log m} \right)^{h(X_i,Y_i)} \geq
m \cdot \left( 1 - \frac{1}{\log m} \right)^{2\log m} 
\left( \frac{1}{2} \right)^{h(X_i,Y_i)} \geq 
\frac{m}{12} \cdot 
\left( \frac{1}{2} \right)^{h(X_i,Y_i)}.$$

Consider the partitions $X_1,\dots,X_{t'}$ and $Y_1,\dots,Y_{t'}$ at the end of the process. 
By \eqref{eq:splitting} and the above, we have $\sum_{i=1}^{t'}{e(X_i,Y_i)} \ge m/12$. 
This completes the proof of the lemma.
\end{proof}

The final, main lemma of this subsection is Lemma \ref{lem:close-length chords} below. Before stating it, we will prove some auxiliary lemmas that will feature in its proof. 
We start with the following lemma.

\begin{lem}\label{lem:interlacing_parallel splitting}
Let $X,Y$ be a section-pair with $m$ chords, and suppose that each vertex of $X$ is incident to at most one chord. 
Then either there is a family of $m/2$ pairwise-interlacing chords, or we can find pairs of chords $(e_i,e'_i)$, $e_i = (x_i,y_i), e'_i = (x'_i,y'_i)$, $i = 1,\dots,t$, such that the following holds:
\begin{enumerate}
    \item For every $1 \leq i \leq t$, $e_i,e'_i$ are either parallel or share a vertex in $Y$.
    \item For every $1 \leq i < j \leq t$, both $e_i$ and $e'_i$ interlace both $e_j$ and $e'_j$.
    \item $\sum_{i=1}^t d_X(x_i,x'_i) \geq m/4$. 
\end{enumerate}
\end{lem}
\begin{proof}
We execute the following process. To start, set $E_1 = E(X,Y)$. For each $i \geq 1$, we do as follows. If every two chords in $E_i$ are interlacing, then stop. Otherwise, among all pairs of chords in $E_i$ which are parallel or share a vertex, choose one pair $e_i,e'_i \in E_i$, $e_i = (x_i,y_i)$, $e'_i = (x'_i,y'_i)$, which maximizes $d_X(x_i,x'_i)$. Without loss of generality, let us assume that $x_i$ is above $x'_i$ (note that $x_i,x'_i$ are distinct by our assumption that every vertex in $X$ is incident to at most one chord; $y_i,y'_i$ may be equal). Since $e_i,e'_i$ are parallel or share a vertex, it holds that either $y_i = y'_i$ or $y_i$ is above $y'_i$. 
Let $X^*$ be the set of vertices of $X$ between $x_i$ and $x'_i$ (including $x_i,x'_i$), and let $Y^*$ be the set of vertices of $Y$
between $y_i$ and $y'_i$ (including $y_i,y'_i$). Observe that if $e = (x,y) \in E_i$ is a chord with $y \in Y^*$, then $x$ must be in $X^*$. Indeed, if $x$ is above $x_i$ then $e,e'_i$ are parallel or share a vertex and $d_X(x,x'_i) > d_X(x_i,x'_i)$; similarly, if $x$ is below $x'_i$ then $e,e_i$ are parallel or share a vertex and $d_X(x,x_i) > d_X(x_i,x'_i)$. In both cases, we get a contradiction to the maximality of $e_i,e'_i$. 

Let $F$ be the set of chords $(x,y) \in E_i$ with $x \in X^*$. Clearly $|F| \geq 2$ because $e_i,e'_i \in F$. Set $E_{i+1} := E_i \setminus F$. 
Note that $d_X(x_i,x'_i) = |X^*| - 1 \geq |F| - 1 \geq |F|/2 = (|E_i| - |E_{i+1}|)/2$, where the first inequality holds because every vertex of $X^*$ is incident to at most one chord. 
Above we have shown that $F$ contains all chords $(x,y)$ with $y \in Y^*$. It is now easy to see that each $e = (x,y) \in E_{i+1}$ interlaces both $e_i$ and $e'_i$. Indeed, let $e = (x,y) \in E_{i+1}$. Then $x$ is either above $x_i$ or below $x'_i$, and $y$ is either above $y_i$ or below $y'_i$. If $x$ is above $x_i$ and $y$ is above $y_i$, then $e,e'_i$ are parallel with $d_X(x,x_i) > d(x_i,x'_i)$, contradicting the maximality of $e_i,e'_i$. Similarly, it is impossible that $x$ is below $x'_i$ and $y$ is below $y'_i$. In the remaining two cases, $e$ interlaces $e_i,e'_i$. This ensures that Item 2 will be satisfied.

Suppose that the process continues for $t$ steps. Then in $E_{t+1}$, every two chords are interlacing. So if $|E_{t+1}| \geq m/2$ then we are done. Suppose then that $|E_{t+1}| < m/2$. Consider the pairs of chords $e_i = (x_i,y_i),e'_i = (x'_i,y'_i)$, $1 \leq i \leq t$. We have already shown that items 1-2 hold. For item 3, we have
$$
\sum_{i = 1}^t {d_X(x_i,x'_i)} \geq \frac{1}{2}\sum_{i = 1}^t (|E_i| - |E_{i+1}|) = \frac{1}{2}(|E_1| - |E_{t+1}|) \geq m/4,
$$
as required. 
\end{proof}

The next two lemmas are concerned with combining paths over a family of interlacing subsections. 

\begin{lem}\label{lem:interlacing two far apart lengths}
Let $X,Y$ be a section-pair with $3t$ interlacing chords, whose lengths belong into an interval of size $D$. 
Then there are paths $P_1,\dots,P_{t}$ between $x^{\text{t}}$ and $y^{\text{t}}$ such that $1 \leq |P_{i+1}| - |P_i| \leq 2D$ for every $1 \leq i \leq t - 1$. In particular, $|P_{t}| \geq |P_1| + t - 1$. 
\end{lem}
\begin{proof}
Let $(x_1,y_1),\dots,(x_{3t},y_{3t}) \in E(X,Y)$ be the given family of interlacing chords, with $x_i$ above $x_j$ and $y_i$ below $y_j$ for each $1 \leq i < j \leq 3t$. For $1 \leq i \leq 3t-1$, put $d_i := d_Y(y_i,y_{i+1}) - d_X(x_i,x_{i+1})$. Note that $d_i$ is precisely the difference between the lengths of $(x_i,y_i)$ and $(x_{i+1},y_{i+1})$, so by our assumption, we have $|d_i|\le D-1.$
We prove the lemma by induction on $t$. 
In the base case $t=1$, the desired path $P_1$ can be chosen as the trivial path of the chord $(x_1,y_1)$.

For the induction step, we will first find two paths $Q_1,Q_2$ between $x_1$ and $x_3$ such that $1 \leq |Q_2| - |Q_1| \leq 2D$, and such that $Q_1,Q_2$ do not use any vertices of $X$ below $x_3$ or any vertices of $Y$ above $y_3$. To find such $Q_1,Q_2$, we proceed as follows.  
For $i = 1,2$, put $a_i := d_X(x_i,x_{i+1})$ and $b_i := d_Y(y_i,y_{i+1})$, so that $d_i = b_i - a_i$. 
Consider the following four paths between $x_1$ and $x_3$.
\begin{itemize}
    \item $P'_1 = X[x_1,x_3]$, \hspace{4.65cm} depicted in blue\hspace{0.3cm} in \Cref{fig:0.4};
    \item $P'_2 = X[x_1,x_2],(x_2,y_2),Y[y_2,y_3],(y_3,x_3)$, \hspace{0.4cm}depicted in red\hspace{0.5cm} in \Cref{fig:0.4};
    \item $P'_3 = (x_1,y_1),Y[y_1,y_2],(y_2,x_2),X[x_2,x_3]$, \hspace{0.4cm}depicted in green\hspace{0.15cm} in \Cref{fig:0.4};
    \item $P'_4 = (x_1,y_1),Y[y_1,y_3],(y_3,x_3)$, \hspace{1.95cm} depicted in yellow in \Cref{fig:0.4}.
\end{itemize}

\begin{figure}
\RawFloats
\begin{minipage}[t]{0.45\textwidth}
\centering
\captionsetup{width=\textwidth}
\begin{tikzpicture}[scale=0.8]

\defPt{0}{0}{x_b}
\defPt{0}{4.5}{x_t}
\defPtm{($0.17*(x_b)+0.83*(x_t)$)}{x1}
\defPtm{($0.5*(x_b)+0.5*(x_t)$)}{x2}

\defPtm{($0.83*(x_b)+0.17*(x_t)$)}{x3}


\defPt{3.5}{0}{y_b}
\defPt{3.5}{4.5}{y_t}


\defPtm{($0.17*(y_b)+0.83*(y_t)$)}{y3}

\defPtm{($0.5*(y_b)+0.5*(y_t)$)}{y2}
\defPtm{($0.83*(y_b)+0.17*(y_t)$)}{y1}



\draw[line width= 0.75 pt] (x_b) -- (x3);
\draw[line width= 0.75 pt] (y_b) -- (y1);
\draw[line width= 0.75 pt] (x_t) -- (x1);
\draw[line width= 0.75 pt] (y_t) -- (y3);

\foreach \i in {1,...,3}
{
\node[] at ($(x\i)+(-0.5,0)$) {\small $x_{\i}$};
\node[] at ($(y\i)+(0.5,0)$) {\small $y_{\i}$};
}

\node[] at ($0.5*(x1)+0.5*(x2)+(-1,0)$) {\small $a_{1}$};
\node[] at ($0.5*(x3)+0.5*(x2)+(-1,0)$) {\small $a_{2}$};

\node[] at ($0.5*(y1)+0.5*(y2)+(1,0)$) {\small $b_{1}$};
\node[] at ($0.5*(y3)+0.5*(y2)+(1,0)$) {\small $b_{2}$};

\draw[line width= 2 pt, blue] ($(x1)+(-0.07,0)$) -- ($(x3)+(-0.07,0)$);

\draw[line width= 2 pt, red] ($(x1)+(0.07,0)$) -- ($(x2)+(0.07,0)$);
\draw[line width= 2 pt, red] ($(x2)+(0,0.07)$) -- ($(y2)+(0,0.07)$);
\draw[line width= 2 pt, red] ($(x3)+(0,-0.07)$) -- ($(y3)+(0,-0.07)$);
\draw[line width= 2 pt, red] ($(y3)+(-0.07,0)$) -- ($(y2)+(-0.07,0)$);

\draw[line width= 2 pt, green] ($(x2)+(0.07,0)$) -- ($(x3)+(0.07,0)$);
\draw[line width= 2 pt, green] ($(x1)+(0,0.07)$) -- ($(y1)+(0,0.07)$);
\draw[line width= 2 pt, green] ($(y1)+(-0.07,0)$) -- ($(y2)+(-0.07,0)$);
\draw[line width= 2 pt, green] ($(x2)+(0,-0.07)$) -- ($(y2)+(0,-0.07)$);

\draw[line width= 2 pt, yellow] ($(x1)+(0,-0.07)$) -- ($(y1)+(0,-0.07)$);
\draw[line width= 2 pt, yellow] ($(y1)+(0.07,0)$) -- ($(y3)+(0.07,0)$);
\draw[line width= 2 pt, yellow] ($(x3)+(0,0.07)$) -- ($(y3)+(0,0.07)$);

\node[] at ($(x_t)+(-0.5,0.1)$) {\small $x^{\text{t}}$};
\node[] at ($(x_b)+(-0.5,0)$) {\small $x^{\text{b}}$};
\node[] at ($(y_t)+(0.5,0.1)$) {\small $y^{\text{t}}$};
\node[] at ($(y_b)+(0.5,0)$) {\small $y^{\text{b}}$};

\draw[] (x_b) \smvx;
\draw[] (x_t) \smvx;
\foreach \i in {1,3}
{
\draw[] (x\i) \vx;
\draw[] (y\i) \vx;
}
\draw[] (y_b) \smvx;
\draw[] (y_t) \smvx;

\draw[] (x2) \smvx;
\draw[] (y2) \smvx;

\end{tikzpicture}
\caption{Different paths $P'_i$.}
\label{fig:0.4}
\end{minipage}\hfill%
\begin{minipage}[t]{0.49\textwidth}
\centering
\captionsetup{width=\textwidth}
\begin{tikzpicture}[scale=0.8]

\defPt{0}{0}{x_b}
\defPt{0}{4.5}{x_t}
\defPtm{($0.1*(x_b)+0.9*(x_t)$)}{x1}
\defPtm{($0.3*(x_b)+0.7*(x_t)$)}{x2}

\defPtm{($0.5*(x_b)+0.5*(x_t)$)}{x3}
\defPtm{($0.7*(x_b)+0.3*(x_t)$)}{x4}

\defPtm{($0.775*(x_b)+0.225*(x_t)$)}{x5}
\defPtm{($0.85*(x_b)+0.15*(x_t)$)}{x6}
\defPtm{($0.925*(x_b)+0.075*(x_t)$)}{x7}

\defPt{3.5}{0}{y_b}
\defPt{3.5}{4.5}{y_t}

\defPtm{($0.075*(y_b)+0.925*(y_t)$)}{y7}
\defPtm{($0.15*(y_b)+0.85*(y_t)$)}{y6}
\defPtm{($0.225*(y_b)+0.775*(y_t)$)}{y5}

\defPtm{($0.3*(y_b)+0.7*(y_t)$)}{y4}
\defPtm{($0.5*(y_b)+0.5*(y_t)$)}{y3}

\defPtm{($0.7*(y_b)+0.3*(y_t)$)}{y2}
\defPtm{($0.9*(y_b)+0.1*(y_t)$)}{y1}

\draw[line width= 0.75 pt] (x_b) -- (x_t);
\draw[line width= 0.75 pt] (y_b) -- (y_t);

\draw[line width= 2 pt] (x4) -- (x_b);

\draw[line width= 2 pt] (y4) -- (y_t);

\foreach \i in {1,...,7}
{
\draw[line width= 1 pt] (x\i) -- (y\i);
}
\foreach \i in {1,...,4}
{
\node[] at ($(x\i)+(-0.5,0.1)$) {\small $x_{\i}$};
\node[] at ($(y\i)+(0.5,0.1)$) {\small $y_{\i}$};
}

\node[] at ($0.5*(x5)+0.5*(x7)+(-0.5,0.1)$) {\small $\vdots$};
\node[] at ($0.5*(y5)+0.5*(y7)+(0.5,0.1)$) {\small $\vdots$};

\draw[spring, line width =2pt, red] (x4) to[out=-0,in=-120] (y7);
\draw[line width =2pt, red] (x_t) -- (x2);
\draw[line width =2pt, red] (x2) -- (y2);
\draw[line width =2pt, red] (y3) -- (y2);
\draw[line width =2pt, red] (y3) -- (x3);
\draw[line width =2pt, red] (x3) -- (x4);
\draw[line width =2pt, red] (y7) -- (y_t);

\node[] at ($(x_t)+(-0.5,0.1)$) {\small $x^{\text{t}}$};
\node[] at ($(x_b)+(-0.5,0)$) {\small $x^{\text{b}}$};
\node[] at ($(y_t)+(0.5,0.1)$) {\small $y^{\text{t}}$};
\node[] at ($(y_b)+(0.5,0)$) {\small $y^{\text{b}}$};

\draw[] (x_b) \smvx;
\draw[] (x_t) \smvx;
\foreach \i in {1,...,4}
{
\draw[] (x\i) \smvx;
\draw[] (y\i) \smvx;
}
\draw[] (y_b) \smvx;
\draw[] (y_t) \smvx;

\end{tikzpicture}
\caption{A combination of $P_2'$ and a path provided by induction (marked as zigzag in the figure).}
\label{fig:0.5}
\end{minipage}
\end{figure}

Observe that $P'_1,\dots,P'_4$ do not use any vertices of $X$ below $x_3$ or any vertices of $Y$ above $y_3$.
We have $|P'_1| = a_1 + a_2$, $|P'_2| = a_1 + b_2 + 2$, $|P'_3| = a_2 + b_1 + 2$, $|P'_4| = b_1 + b_2 + 2$. 
Hence, $|P'_2| - |P'_1| = d_2 + 2$, $|P'_3| - |P'_1| = d_1 + 2$, $|P'_4| - |P'_1| = d_1 + d_2 + 2$, $|P'_3| - |P'_2| = d_1 - d_2$, $|P'_4| - |P'_2| = d_1$ and $|P'_4| - |P'_3| = d_2$. So we see that for all $Q_1,Q_2 \in \{P'_1,\dots,P'_4\}$, it holds that $||Q_2| - |Q_1|| \leq |d_1| + |d_2| + 2 \leq 2D$.  

Observe that the four numbers $a_1 + a_2, a_1 + b_2 + 2, a_2 + b_1 + 2, b_1 + b_2 + 2$ cannot all be equal. Indeed, if the first three are equal then $b_1 = a_1 - 2$ and $b_2 = a_2 - 2$, which means that $b_1 + b_2 + 2 = a_1 + a_2 - 2 \neq a_1 + a_2$. It follows that there are $Q_1,Q_2 \in \{P'_1,\dots,P'_4\}$ with $|Q_1| \neq |Q_2|$, say $|Q_1| < |Q_2|$. We have already shown that $|Q_2| - |Q_1| \leq 2D$. So $Q_1,Q_2$ satisfy all of our requirements.

We now complete the induction step using the paths $Q_1,Q_2$ found above. Let $t \geq 2$, and apply the induction hypothesis with parameter $t - 1$ and with $X$ and $Y$ 
replaced by $X[x_4,x^{\text{b}}]$ and $Y[y^{\text{t}},y_4]$, which still contain $3(t-1)$ of our interlacing chords. This way we obtain paths $R_1,\dots,R_{t-1}$ between $x_4$ and $y^{\text{t}}$ such that $1 \leq |R_{i+1}| - |R_i| \leq 2D$ for every $1 \leq i \leq t - 2$. 
For $1 \leq i \leq t-1$, let $P_i$ be the concatenation of $X[x^{\text{t}},x_1]$, $Q_1$, $X[x_3,x_4]$ and $R_i$. Let $P_{t}$ be the concatenation of $X[x^{\text{t}},x_1]$, $Q_2$, $X[x_3,x_4]$ and $R_{t-1}$. See \Cref{fig:0.5} for an illustration of the resulting paths $P_i$.
For each $1 \leq i \leq t-2$, we have $|P_{i+1}| - |P_i| = |R_{i+1}| - |R_i|$ and hence 
$1 \leq |P_{i+1}| - |P_i| \leq 2D$. Lastly, $|P_{t}| - |P_{t - 1}| = |Q_2| - |Q_1|$, and hence $1 \leq |P_{t}| - |P_{t - 1}| \leq 2D$ as well. This completes the proof. 
\end{proof}

\begin{lem}\label{lem:interlacing_parallel two far apart lengths}
Let $X,Y$ be a section-pair, let $t \geq 1$ be odd, and let $X_1,Y_1;\dots;X_t,Y_t$ be an interlacing collection of subsection pairs. 
Let $d_1,\dots,d_t \geq 1$. 
Suppose that for each $1 \leq i \leq t$, there are two paths $Q_i^R,Q_i^B \subseteq X_i \cup Y_i$ with $|Q_i^R| - |Q_i^B| \geq d_i$, which go between $x^{\text{t}}_i$ and $y^{\text{t}}_i$ if $i$ is odd and between $x^{\text{b}}_i$ and $y^{\text{b}}_i$ if $i$ is even. Set $D := \max_{i = 1,\dots,t}{(|X_i| + |Y_i|)}$. Then there are paths $P_1,\dots,P_{(t+3)/2}$ between $x^{\text{t}}$ and $y^{\text{t}}$ such that $|P_{(t+3)/2}| - |P_1| \geq \sum_{i = 1}^t{d_i}$ and such that $1 \leq |P_{i+1}| - |P_i| \leq 2D$ for every $1 \leq i \leq (t+1)/2$. 
\end{lem}

\begin{proof}
Without loss of generality, assume that $X_i$ is above $X_j$ for all $i<j$ (and consequently, the opposite is true for the $Y_i$'s).
The proof is by induction on $t$. 
For the case $t = 1$, simply take $P_1/P_2$ to be the concatenation of $X[x^{\text{t}},x^{\text{t}}_1]$, $Q_1^B/Q_1^R$ and $Y[y^{\text{t}}_1,y^{\text{t}}]$. Clearly, $|P_2| - |P_1| = |Q_1^R| - |Q_1^B|$, which immediately implies that $|P_2| - |P_1| \geq d_1 \geq 1$. On the other hand, since $Q_1^R,Q_1^B$ are both contained in $X_1 \cup Y_1$, we have $|P_2| - |P_1| = |Q_1^R| - |Q_1^B| \leq |X_1| + |Y_1| \leq D$, as required. 

Suppose now that $t \geq 3$. For each $c = R,B$, define a path $Q'_c$ from $x^{\text{t}}$ to $x_2^{\text{b}}$ as follows:
$$Q'_c = X[x^{\text{t}},x_1^{\text{t}}], Q_1^c, Y[y_1^{\text{t}},y_2^{\text{b}}], Q_2^c.$$

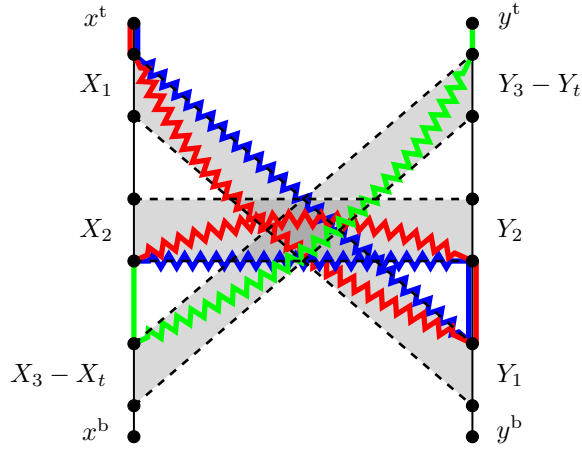
\begin{figure}
\centering
\begin{tikzpicture}
\defPt{0}{0}{x_b}
\defPt{0}{5.5}{x_t}
\defPtm{($0.075*(x_b)+0.925*(x_t)$)}{x6}
\defPtm{($0.225*(x_b)+0.775*(x_t)$)}{x5}

\defPtm{($0.425*(x_b)+0.575*(x_t)$)}{x4}
\defPtm{($0.575*(x_b)+0.425*(x_t)$)}{x3}

\defPtm{($0.775*(x_b)+0.225*(x_t)$)}{x2}
\defPtm{($0.925*(x_b)+0.075*(x_t)$)}{x1}

\defPt{4.5}{0}{y_b}
\defPt{4.5}{5.5}{y_t}

\defPtm{($0.075*(y_b)+0.925*(y_t)$)}{y2}
\defPtm{($0.225*(y_b)+0.775*(y_t)$)}{y1}

\defPtm{($0.425*(y_b)+0.575*(y_t)$)}{y4}
\defPtm{($0.575*(y_b)+0.425*(y_t)$)}{y3}

\defPtm{($0.775*(y_b)+0.225*(y_t)$)}{y6}
\defPtm{($0.925*(y_b)+0.075*(y_t)$)}{y5}

\draw[color=white, fill=black!50!,fill opacity=0.3] (x1) -- (x2) -- (y2) -- (y1) -- cycle;
\draw[color=white, fill=black!50!,fill opacity=0.3] (x3) -- (x4) -- (y4) -- (y3) -- cycle;
\draw[color=white, fill=black!50!,fill opacity=0.3] (x5) -- (x6) -- (y6) -- (y5) -- cycle;

\draw[line width =2pt, blue] ($(x_t)+(0.05,0)$) -- ($(x6)+(0.05,0)$);
\draw[line width =2pt, red] ($(x_t)+(-0.05,0)$) -- ($(x6)+(-0.05,0)$);

\draw[line width =2pt, blue] ($(y3)+(-0.05,0)$) -- ($(y6)+(-0.05,0)$);
\draw[line width =2pt, red] ($(y3)+(0.05,0)$) -- ($(y6)+(0.05,0)$);

\draw[line width= 0.75 pt] (x_b) -- (x_t);
\draw[line width= 0.75 pt] (y_b) -- (y_t);

\draw[spring, line width =2pt, blue] (x6) -- (y6);
\draw[spring, line width =2pt, red] (x6) to[out=-65, in=155] (y6);

\draw[spring, line width =2pt, blue] (x3) -- (y3);
\draw[spring, line width =2pt, red] (x3) to[out=25, in=-205] (y3);

\draw[line width =2pt, green] (x3) -- (x2);
\draw[spring, line width =2pt, green] (x2) to[out=25, in=-115] (y2);
\draw[line width =2pt, green] (y2) -- (y_t);

\draw[] (x_b) \smvx;
\draw[] (x_t) \smvx;
\foreach \i in {1,...,6}
{
\draw[] (x\i) \smvx;
\draw[] (y\i) \smvx;
}
\draw[] (y_b) \smvx;
\draw[] (y_t) \smvx;



\foreach \i in {1,...,6}
{
\draw[line width= 1 pt, dashed] (x\i) -- (y\i);
}




\node[] at ($0.5*(x1)+0.5*(x2)+(-1,0)$) {\small $X_3-X_t$};
\node[] at ($0.5*(x3)+0.5*(x4)+(-0.5,0)$) {\small $X_2$};
\node[] at ($0.5*(x5)+0.5*(x6)+(-0.5,0)$) {\small $X_1$};

\node[] at ($0.5*(y1)+0.5*(y2)+(0.9,0)$) {\small $Y_3-Y_{t}$};
\node[] at ($0.5*(y3)+0.5*(y4)+(0.5,0)$) {\small $Y_2$};
\node[] at ($0.5*(y5)+0.5*(y6)+(0.5,0)$) {\small $Y_1$};

\node[] at ($(x_t)+(-0.5,0.1)$) {\small $x^{\text{t}}$};
\node[] at ($(x_b)+(-0.5,0.1)$) {\small $x^{\text{b}}$};
\node[] at ($(y_t)+(0.5,0.1)$) {\small $y^{\text{t}}$};
\node[] at ($(y_b)+(0.5,0.1)$) {\small $y^{\text{b}}$};

\end{tikzpicture}
\captionsetup{width=0.85\textwidth}
\caption{$Q_1^R,Q_2^R$ are marked as red zigzag paths, $Q'_R$ is depicted in red, $Q_1^B,Q_2^B$ are marked as blue zigzag paths, $Q'_B$ is depicted in blue, and some $R_i$ is depicted in green.}
\label{fig:0.6}
\end{figure}

See \Cref{fig:0.6} for an illustration. 
We have
$|Q'_R| - |Q'_B| = \sum_{i = 1}^2(|Q_i^R| - |Q_i^B|) \geq d_1 + d_2$. 
On the other hand, since $Q^R_i,Q^B_i \subseteq X_i \cup Y_i$ (for $i = 1,2$), we have that $|Q^R_i| - |Q^B_i| \leq D$ and hence $|Q'_R| - |Q'_B| \leq 2D$. Note that $Q'_B,Q'_R$ do not use any vertices of $X$ below $x_2^{\text{b}}$ or any vertices of $Y$ above $y_2^{\text{t}}$. 
Now apply the induction hypothesis with parameter $t-2$ and with $X$ and $Y$ replaced by $X[x_2^{\text{b}},x^{\text{b}}]$ and $Y[y^{\text{t}},y_3^{\text{b}}]$. This gives paths 
$R_1,\dots,R_{(t+1)/2}$ between $x_2^{\text{b}}$ and $y^{\text{t}}$ such that $|R_{(t+1)/2}| - |R_1| \geq \sum_{i = 3}^t{d_i}$ and such that $1 \leq |R_{i+1}| - |R_i| \leq 2D$ for every $1 \leq i \leq (t-1)/2$. 
For $1 \leq i \leq (t+1)/2$, let $P_i$ be the concatenation of $Q'_B$ and $R_i$. Let $P_{(t+3)/2}$ be the concatenation of $Q'_R$ and $R_{(t+1)/2}$. Then  
$|P_{(t+3)/2}| - |P_1| = |Q'_R| - |Q'_B| + |R_{(t+1)/2}| - |R_1| \geq \sum_{i = 1}^t{d_i}$. Moreover, for every $1 \leq i \leq (t-1)/2$ we have $|P_{i+1}| - |P_i| = |R_{i+1}| - |R_i|$ and hence $1 \leq |P_{i+1}| - |P_i| \leq 2D$. 
Finally, we have $|P_{(t+3)/2}| - |P_{(t+1)/2}| = |Q'_R| - |Q'_B|$ and hence $1 \leq |P_{(t+3)/2}| - |P_{(t+1)/2}| \leq 2D$. This \nolinebreak completes \nolinebreak the \nolinebreak proof.
\end{proof}

Using Lemmas \ref{lem:interlacing_parallel splitting}, \ref{lem:interlacing two far apart lengths} and \ref{lem:interlacing_parallel two far apart lengths}, we can now prove the following lemma, which states that in a section pair with $m$ chords one can find two paths whose lengths differ by $\Omega(m)$, provided every vertex on one side is incident to at most one chord. 

\begin{lem}\label{lem:distant_paths}
Let $X,Y$ be a section-pair with $m \geq 12$ chords, and suppose that each vertex of $X$ is incident to at most one chord. Then there are two paths $P,P'$ between $x^{\text{t}}$ and $y^{\text{t}}$ such that $|P'| \geq |P| + \Omega(m)$. 
By symmetry, there are two paths $Q,Q'$ between $x^{\text{b}}$ and $y^{\text{b}}$ such that $|Q'| \geq |Q| + \Omega(m)$. 
\end{lem}

\begin{proof}
Apply Lemma \ref{lem:interlacing_parallel splitting}. If $E(X,Y)$ contains a family of $m/2$ pairwise-interlacing chords, then we are done by Lemma \ref{lem:interlacing two far apart lengths}. Otherwise, let $e_i = (x_i,y_i)$, $e'_i = (x'_i,y'_i)$, $1 \leq i \leq t$, be as in Lemma \ref{lem:interlacing_parallel splitting}.
Without loss of generality we may assume, for each $1 \leq i \leq t$, that $x_i$ is above $x'_i$, and hence either $y_i = y'_i$ or $y_i$ is above $y'_i$.
For $1 \leq i \leq t$, let $X_i := X[x_i,x'_i]$, $Y_i := Y[y_i,y'_i]$; so $x_i^{\text{t}} = x_i$, $x_i^{\text{b}} = x'_i$, $y_i^{\text{t}} = y_i$, and $y_i^{\text{b}} = y'_i$. Item 2 of Lemma \ref{lem:interlacing_parallel splitting} guarantees that, after possibly permuting the coordinates $1,\dots,t$, we have that $X_i$ is above $X_j$ and $Y_i$ is below $Y_j$ for every $1 \leq i < j \leq t$. 
Set $d_i := d_X(x_i,x'_i)$.
Observe that for each $1 \leq i \leq t$, $Q^B_i := (x_i,y_i)$ and $Q^R_i := X[x_i,x'_i], (x'_i,y'_i), Y[y'_i,y_i]$ are two paths between $x_i$ and $y_i$, both contained in $X_i \cup Y_i$, such that $|Q^R_i| - |Q^B_i| = d_X(x_i,x'_i) + d_Y(y_i,y'_i) \geq d_i$. Similarly, $Q^B_i := (x'_i,y'_i)$ and $Q^R_i := X[x'_i,x_i], (x_i,y_i), Y[y_i,y'_i]$ are two paths between $x'_i$ and $y'_i$, both contained in $X_i \cup Y_i$, such that $|Q^R_i| - |Q^B_i| \geq d_i$.
Hence, we are in the setting of Lemma \ref{lem:interlacing_parallel two far apart lengths}. 
By item 3 of Lemma \ref{lem:interlacing_parallel splitting}, we have $\sum_{i=1}^t {d_i} \geq m/4$. 
If $t$ is even, then omit a pair $(X_i,Y_i)$ for which $d_i = \min\{d_1,\dots,d_t\}$; this way we guarantee that $t$ is odd (which is necessary for invoking Lemma \ref{lem:interlacing_parallel two far apart lengths}), while making sure that the sum $\sum d_i$ over the remaining $i$ is at least $m/8$. 
To complete the proof, apply Lemma \ref{lem:interlacing_parallel two far apart lengths} and set $P := P_1$ \nolinebreak and \nolinebreak $P' := P_{(t+3)/2}$. 
\end{proof}

Finally, we are in a position to prove Lemma \ref{lem:close-length chords}. This lemma states that if a section-pair has many chords whose lengths are close to each other, say all contained in some (small) interval $J$, then one can find a sequence of path lengths, with any two consecutive lengths differing by $\Theta(|J|)$. This lemma will play a key role in the proof of Theorem \ref{thm:main}. 

\begin{lem}\label{lem:close-length chords}
Let $X,Y$ be a section-pair, and suppose that each vertex of $X$ is incident to at most one chord. Let $E \subseteq E(X,Y)$ be a set of chords, and let $J \subseteq \mathbb{N}$ be an interval such that for each $e \in E$, the length of $e$ belongs to $J$.  
Then for $k = \Omega(|E|/|J|)$, there are paths $P_1,\dots,P_{k}$ between $x^{\text{t}}$ and $y^{\text{t}}$ such that 
$|P_{i+1}| - |P_i| = \Theta(|J|)$ for every $1 \leq i \leq k-1$. 
\end{lem}
\begin{proof}
We will assume that $|E| \gg |J|$, as otherwise, the assertion is trivial.
We only consider the chords in $E$ (removing all others). 
Partition $X$ into subsections $X_1,\dots,X_s$, such that $X_1,\dots,X_{s-1}$ have $|J|$ vertices each, and $X_s$ has less than $|J|$ vertices. 
(As usual, $X_i$ is below $X_j$ for $i < j$.)
Note that if $i \neq j$ have the same parity, then for all $x_1 \in X_i, x_2 \in X_j$ we have $d_X(x_1,x_2) > |J|$. 
Without loss of generality, we shall assume that 
$\sum_{i \text{ even}}{e(X_i)} \geq \sum_{i \text{ odd}}{e(X_i)}$, and hence
$\sum_{i \text{ even}}{e(X_i)} \geq |E|/2$ (the other case is symmetrical).
For convenience, let us set $I_0 = \{i \in [s] : i \text{ even}\}$. For each $i \in I_0$, let $Y_i$ be the minimal subpath of $Y$ containing all neighbours of $X_i$.
Note that $\sum_{i \in I_0}{e(X_i,Y_i)} \geq |E|/2$.

Observe that for all $(x_1,y_1),(x_2,y_2) \in E(X_i,Y_i)$, we have $d_Y(y_1,y_2) \leq d_X(x_1,x_2) + |J| < 2|J|$, where the first inequality follows from the fact that the lengths of $(x_1,y_1)$ and $(x_2,y_2)$ both belong to $J$, and hence differ by at most $|J|$. 
So we see that $|Y_i| \leq 2|J|$. 
Next, observe that for every pair $i,j \in I_0$ with $i < j$ and for all $(x_1,y_1) \in E(X_i,Y_i)$, $(x_2,y_2) \in E(X_j,Y_j)$, it must be the case that $y_2$ is below $y_1$. Indeed, if not, then $(x_1,y_1),(x_2,y_2)$ either are parallel or share a vertex, which means that the difference between the lengths of $(x_1,y_1)$ and $(x_2,y_2)$ is at least $d_X(x_1,x_2) > |J|$, contradicting that both of these lengths are in $J$. So we see that $Y_j$ is above $Y_i$. In other words, every chord in $E(X_i,Y_i)$ interlaces every chord \nolinebreak in \nolinebreak $E(X_j,Y_j)$.  

Let $I'$ (resp. $I''$) be the set of all $i \in I_0$ with $0 < e(X_i,Y_i) < 12$ (resp. $e(X_i,Y_i) \geq 12$). Suppose first that $\sum_{i \in I'}{e(X_i,Y_i)} \geq |E|/4$. Then by picking one chord from $E(X_i,Y_i)$ for each $i \in I'$, we obtain a family of pairwise-interlacing chords of size $\Omega(|E|)$. 
Let $(x_1,y_1),\dots,(x_{3t},y_{3t})$ be such a family with $t = \Omega(|E|)$.
By Lemma \ref{lem:interlacing two far apart lengths}, there are paths
$P_1,\dots,P_{t}$ between $x^{\text{t}}$ and $y^{\text{t}}$ such that $1 \leq |P_{i+1}| - |P_i| \leq 2|J|$ for every $1 \leq i \leq t - 1$. It is easy to see that under these conditions, we can find indices $1 \leq i_1 < \dots < i_k$ with $k = \Omega(t/|J|) = \Omega(|E|/|J|)$, such that $|P_{i_{j+1}}| - |P_{i_j}| = \Theta(|J|)$ for every $1 \leq j \leq k-1$. Thus, the assertion of the lemma holds in this case.

Suppose now that $\sum_{i \in I''}{e(X_i,Y_i)} \geq |E|/4$. Write $I'' = \{i_1,\dots,i_t\}$, where $i_1 < \dots < i_t$. 
If $t$ is even, then we only consider $i_1,\dots,i_{t-1}$, thus making sure that the number of indices in consideration is odd. Observe that $e(X_{i_t},Y_{i_t}) \leq |X_{i_t}| \leq |J|$ (since every vertex of $X$ is incident to at most one chord), and hence $\sum_{j=1}^{t-1}{e(X_{i_j},Y_{i_j})} \geq |E|/4 - |J| \geq |E|/8$. So with a slight abuse of notation, we will assume from now on that $t$ is odd and $\sum_{j=1}^{t}{e(X_{i_j},Y_{i_j})} \geq |E|/8$.  

Fix any $1 \leq j \leq t$. Applying Lemma \ref{lem:distant_paths} to $X_{i_j}$ and $Y_{i_j}$, gives paths $P,P' \subseteq X_{i_j} \cup Y_{i_j}$ between $x_{i_j}^{\text{t}}$ and $y_{i_j}^{\text{t}}$ with $|P'| \geq |P| + \Omega(e(X_{i_j},Y_{i_j}))$, as well as paths $Q,Q' \subseteq X_{i_j} \cup Y_{i_j}$ between $x_{i_j}^{\text{b}}$ and $y_{i_j}^{\text{b}}$ with $|Q'| \geq |Q| + \Omega(e(X_{i_j},Y_{i_j}))$. 
Hence, we are in the setting of Lemma \ref{lem:interlacing_parallel two far apart lengths} with $d_j = \Omega(e(X_{i_j},Y_{i_j}))$. Moreover, $D := \max_{j = 1,\dots,t}{(|X_{i_j}| + |Y_{i_j}|)} \leq 3|J|$. By Lemma \ref{lem:interlacing_parallel two far apart lengths}, there are paths $P_1,\dots,P_{(t+3)/2}$ between $x^{\text{t}}$ and $y^{\text{t}}$ such that $|P_{(t+3)/2}| - |P_1| \geq \sum_{j = 1}^t{d_j} = \sum_{j = 1}^t{\Omega(e(X_{i_j},Y_{i_j}))} = \Omega(|E|)$, and such that $1 \leq |P_{i+1}| - |P_i| \leq 2D = O(|J|)$ for every $1 \leq i \leq (t+1)/2$. Again, it is easy to see that under these conditions, we can find indices $1 \leq i_1 < \dots < i_k$ with $k = \Omega(|E|/|J|)$, such that $|P_{i_{j+1}}| - |P_{i_j}| = \Theta(|J|)$ for every $1 \leq j \leq k-1$. So we have established the assertion of the lemma in this case as well. This completes the proof. 
\end{proof}

\subsection{Putting it all together}
The following is the main lemma in the proof of Theorem \ref{thm:main}.

\begin{lem}\label{lem:main-part}
Let $X,Y$ be a section-pair, and suppose that each vertex of $X$ is incident to at most one chord. Let $X_1,Y_1;\ldots; X_t,Y_t$ be a parallel collection of subsection pairs such that for every $i$, there are $m$ chords between $X_i$ and $Y_i$. Provided $ t \ge \sqrt{\log m}$, there are at least $\Omega(m/2^{4\sqrt{\log m \log \log m}})$ different lengths of paths from $x^{\text{t}}$ to $y^{\text{t}}$.
\end{lem}

\begin{proof}
We may and will assume throughout that $m$ is large enough. Let us for convenience define $\eps:=\sqrt{\frac{\log \log m}{\log m}}$. Observe that $m^{\eps}=2^{\sqrt{\log m \log \log m}}$, so our goal is to find $\Omega(m^{1-4\eps})$ different lengths of paths from $x^{\text{t}}$ to $y^{\text{t}}$. 
Let us assume towards a contradiction that there are less than $\Omega(m^{1-4\eps})$ different lengths of paths from $x^{\text{t}}$ to $y^{\text{t}}$.
Recall that $x_i^{\text{t}}$/$x_i^{\text{b}}$ denotes the top/bottom vertex of $X_i$, and similarly $y_i^{\text{t}}$/$y_i^{\text{b}}$ denotes the top/bottom vertex of $Y_i$.
We will prove the following claim by induction on $i$. 

\textbf{Induction statement.} For a large enough constant $C>0$ and for every $i \le 1/\eps$, there are below-paths from $x_i^{\text{t}}$ to $y_i^{\text{t}}$ of $\frac{m^{i\eps}}{6(C\log m)^{2(i-1)}}$ different lengths, all of which belong to an interval of size $m^{i\eps}$. 

Before turning to the proof of this statement, let us show that it indeed implies the lemma. By choosing $i=\floor{1/\eps}$, we have $m^{i\eps} \ge m^{1-\eps}$ and 
\begin{equation}\label{eq:log-epsilons}
    (\log m)^{1/\eps}=2^{\sqrt{\log m \log \log m}}=m^{\eps}.
\end{equation} 
So, using, e.g. that $C^2 \leq \log m$, we get 
$$
\frac{m^{i\varepsilon}}{(C\log m)^{2(i-1)}} \geq \frac{m^{1-\varepsilon}}{(\log m)^{3i}} \geq 
\frac{m^{1-\varepsilon}}{(\log m)^{3/\varepsilon}} \geq m^{1-4\varepsilon}.
$$
This implies we find at least $\Omega(m^{1-4\eps})$ path lengths, as desired.

From now on, our goal is to prove the above induction statement. 
For the base case of $i=1$, note that since for any chord $(x,y)$, the trivial path $X[x^{\text{t}},x],(x,y),Y[y,y^{\text{t}}]$ has length exactly the length of $(x,y)$ plus one, there can be at most $m^{1-4\eps} \leq m^{1-\varepsilon}$ different chord lengths. Since there are $m$ chords between $X_1$ and $Y_1$, there must exist a length which appears at least $m^{\eps}$ many times. 
Apply \Cref{lem:interlacing two far apart lengths}  with $t=\lfloor m^{\eps}/3 \rfloor$ to the family of chords having this length, noting that chords of the same length must be interlacing, and that for such chords the parameter $D$ in Lemma \ref{lem:interlacing two far apart lengths} equals $1$. Lemma \ref{lem:interlacing two far apart lengths} gives us a collection of $t \geq m^{\eps}/6$ lengths of paths between $x_1^{\text{t}}$ and $y_1^{\text{t}}$ with all consecutive differences being either $1$ or $2$ apart, so in total belonging into an interval of size at most $2t+1 \leq m^{\eps}$, as desired.

Let us now move on to the induction step. 
So assume that for some $1\le i \le 1/\eps-1$, there are below-paths going from $x_{i}^{\text{t}}$ to $y_{i}^{\text{t}}$ of $\frac{m^{i\eps}}{6(C\log m)^{2(i-1)}}$ different lengths, all of which belong to an interval of size $m^{i\eps}$. Note that we can extend all of these paths to below-paths from $x_{i+1}^{\text{b}}$ to 
$y_{i+1}^{\text{b}}$ by using $X[x_{i}^{\text{t}},x_{i+1}^{\text{b}}]$ and $Y[y_{i}^{\text{t}},y_{i+1}^{\text{b}}]$, and still keep the same property concerning their lengths, as each length is increased by the same number. 

Our goal now is to find in $X_{i+1},Y_{i+1}$ \tbpps{} with many
(roughly $m^{\varepsilon}/\text{poly}\log m$) different lengths, among which any two consecutive ones are $\Theta(m^{i\eps})$ apart. We can use any such \tbpp{} to extend any of our below-paths from $x_{i+1}^{\text{b}}$ to $y_{i+1}^{\text{b}}$ into a below-path from $x_{i+1}^{\text{t}}$ to $y_{i+1}^{\text{t}}$ with the length being the sum of the lengths of our original path and that of the \tbpp{} we used. \Cref{lem:spread-close} will then allow us to ``fill in the gaps'' of size $\Theta(m^{i\varepsilon})$ between the lengths of the \tbpps{} we found, by using the path lengths given by the inductive assumption. This will complete the induction step. See \Cref{fig:induction-step} for an illustration. 

Let us apply \Cref{lem:splitting_process} to the section pair $X_{i+1},Y_{i+1}$ to obtain a parallel collection of subsection pairs $X_1',Y_1'; \ldots; X_{t}',Y_{t}'$ such that $\sum_{j=1}^t e(X_j',Y_j') \ge m/24$ and either for every $j$ we have a vertex in $X_j' \cup Y_j'$ incident to ${e(X_j',Y_j')}/{(6\log m)}$ chords, or for every $j$ there is a chord $(x_j,y_j)$ in $E(X_j',Y_j')$ interlacing at least ${e(X_j',Y_j')}/{(6\log m)}$ chords in $E(X_j',Y_j')$. See Figures \ref{fig:split-case-1} and \ref{fig:split-case-2}. In the former case we get $\Omega(m/\log m)\ge\Omega(m^{1-4\eps})$ chords, any two of which are either parallel or share a vertex, which means that all of these chords have different lengths. This in turn means that we have $\Omega(m^{1-4\eps})$ different lengths of (trivial) paths from $x^{\text{t}}$ to $y^{\text{t}}$, giving us a contradiction. So we may assume we are in the latter case.

\begin{figure}
\RawFloats
\begin{minipage}[t]{0.4\textwidth}
\centering
\captionsetup{width=\textwidth}
\begin{tikzpicture}[scale=1]

\defPt{0}{0}{x_b}
\defPt{0}{4.5}{x_t}

\defPtm{($0.125*(x_b)+0.875*(x_t)$)}{x4}
\defPtm{($0.375*(x_b)+0.625*(x_t)$)}{x3}

\defPtm{($0.525*(x_b)+0.475*(x_t)$)}{x2}
\defPtm{($0.875*(x_b)+0.125*(x_t)$)}{x1}

\defPt{2.5}{0}{y_b}
\defPt{2.5}{4.5}{y_t}

\defPtm{($0.125*(y_b)+0.875*(y_t)$)}{y4}
\defPtm{($0.375*(y_b)+0.625*(y_t)$)}{y3}

\defPtm{($0.525*(y_b)+0.475*(y_t)$)}{y2}
\defPtm{($0.875*(y_b)+0.125*(y_t)$)}{y1}

\draw[color=white, fill=black!50!,fill opacity=0.3] (x1) -- (x2) -- (y2) -- (y1) -- cycle;
\draw[color=white, fill=black!50!,fill opacity=0.3] (x3) -- (x4) -- (y4) -- (y3) -- cycle;

\draw[line width= 0.75 pt] (x_b) -- (x_t);
\draw[line width= 0.75 pt] (y_b) -- (y_t);

\draw[spring, line width= 2 pt, red] (x4) -- (y3);
\draw[spring, line width= 2 pt, red] (y4) -- (x3);
\draw[line width= 2 pt, red] (y2) -- (y3);
\draw[line width= 2 pt, red] (x2) -- (x3);

\draw[spring, line width= 2 pt, red] (x2) to[out=-85, in=-95] (y2);

\foreach \i in {2,...,4}
{
\draw[] (x\i) \smvx;
\draw[] (y\i) \smvx;
}

\foreach \i in {1,...,4}
{
\draw[line width= 1 pt, dashed] (x\i) -- (y\i);
}

\node[] at ($0.6*(x1)+0.4*(x2)+(-0.8,0)$) {\small $X_1-X_i$};
\node[] at ($0.5*(x3)+0.5*(x4)+(-1.05,0)$) {\small $X_{i+1}$};

\node[] at ($0.6*(y1)+0.4*(y2)+(0.75,0)$) {\small $Y_1-Y_i$};
\node[] at ($0.5*(y3)+0.5*(y4)+(1,0)$) {\small $Y_{i+1}$};

\node[] at ($(x2)+(-0.5,-0.1)$) {\small $x_i^{\text{t}}$};
\node[] at ($(x3)+(-0.5,-0.05)$) {\small $x_{i+1}^{\text{b}}$};
\node[] at ($(x4)+(-0.5,0.1)$) {\small $x_{i+1}^{\text{t}}$};

\node[] at ($(y2)+(0.5,-0.1)$) {\small $y_i^{\text{t}}$};
\node[] at ($(y3)+(0.5,-0.05)$) {\small $y_{i+1}^{\text{b}}$};
\node[] at ($(y4)+(0.5,0.1)$) {\small $y_{i+1}^{\text{t}}$};


\end{tikzpicture}
\caption{The lower part of the path is provided by the induction and can be chosen among $\tilde{\Omega}(m^{i\eps})$ paths with lengths in an interval of size $m^{i\eps}$. The upper \tbpp{} for $X_{i+1},Y_{i+1}$ can be chosen among a collection of $\tilde{\Omega}(m^{\eps})$ of them with lengths $m^{i\eps}$ apart.}
\label{fig:induction-step}
\end{minipage}\hfill
\begin{minipage}[t]{0.28\textwidth}
\centering
\captionsetup{width=\textwidth}
\begin{tikzpicture}[scale=0.8]

\defPt{0}{0}{x_b}
\defPt{0}{5.5}{x_t}
\defPtm{($0.075*(x_b)+0.925*(x_t)$)}{x6}
\defPtm{($0.275*(x_b)+0.725*(x_t)$)}{x5}

\defPtm{($0.425*(x_b)+0.575*(x_t)$)}{x4}
\defPtm{($0.625*(x_b)+0.375*(x_t)$)}{x3}

\defPtm{($0.725*(x_b)+0.275*(x_t)$)}{x2}
\defPtm{($0.925*(x_b)+0.075*(x_t)$)}{x1}

\defPt{3.5}{0}{y_b}
\defPt{3.5}{5.5}{y_t}

\defPtm{($0.075*(y_b)+0.925*(y_t)$)}{y6}
\defPtm{($0.275*(y_b)+0.725*(y_t)$)}{y5}

\defPtm{($0.425*(y_b)+0.575*(y_t)$)}{y4}
\defPtm{($0.625*(y_b)+0.375*(y_t)$)}{y3}

\defPtm{($0.725*(y_b)+0.275*(y_t)$)}{y2}
\defPtm{($0.925*(y_b)+0.075*(y_t)$)}{y1}

\draw[color=white, fill=black!50!,fill opacity=0.3] (x1) -- (x2) -- (y2) -- (y1) -- cycle;
\draw[color=white, fill=black!50!,fill opacity=0.3] (x3) -- (x4) -- (y4) -- (y3) -- cycle;
\draw[color=white, fill=black!50!,fill opacity=0.3] (x5) -- (x6) -- (y6) -- (y5) -- cycle;

\draw[] (x_b) \smvx;
\draw[] (x_t) \smvx;
\draw[] (y_b) \smvx;
\draw[] (y_t) \smvx;

\draw[line width= 0.75 pt] (x_b) -- (x_t);
\draw[line width= 0.75 pt] (y_b) -- (y_t);


\foreach \i in {1,...,6}
{
\draw[line width= 1 pt, dashed] (x\i) -- (y\i);
}

\draw[line width= 1 pt] ($0.25*(x1)+0.75*(x2)$) -- ($0.5*(y1)+0.5*(y2)$);
\draw[line width= 1 pt] ($0.5*(x1)+0.5*(x2)$) -- ($0.5*(y1)+0.5*(y2)$);
\draw[line width= 1 pt] ($0.75*(x1)+0.25*(x2)$) -- ($0.5*(y1)+0.5*(y2)$);

\draw[line width= 1 pt] ($0.25*(x3)+0.75*(x4)$) -- ($0.5*(y3)+0.5*(y4)$);
\draw[line width= 1 pt] ($0.5*(x3)+0.5*(x4)$) -- ($0.5*(y3)+0.5*(y4)$);
\draw[line width= 1 pt] ($0.75*(x3)+0.25*(x4)$) -- ($0.5*(y3)+0.5*(y4)$);

\draw[line width= 1 pt] ($0.25*(x5)+0.75*(x6)$) -- ($0.5*(y5)+0.5*(y6)$);
\draw[line width= 1 pt] ($0.5*(x5)+0.5*(x6)$) -- ($0.5*(y5)+0.5*(y6)$);
\draw[line width= 1 pt] ($0.75*(x5)+0.25*(x6)$) -- ($0.5*(y5)+0.5*(y6)$);

\node[] at ($0.5*(x1)+0.5*(x2)+(-0.5,0)$) {\small $X_1'$};
\node[] at ($0.5*(x3)+0.5*(x4)+(-0.5,0)$) {\small $X_2'$};
\node[] at ($0.5*(x4)+0.5*(x5)+(-0.5,0.15)$) {\small $\vdots$};
\node[] at ($0.5*(x5)+0.5*(x6)+(-0.5,0)$) {\small $X_{t'}'$};

\node[] at ($0.5*(y1)+0.5*(y2)+(0.5,0)$) {\small $Y_1'$};
\node[] at ($0.5*(y3)+0.5*(y4)+(0.5,0)$) {\small $Y_2'$};
\node[] at ($0.5*(y4)+0.5*(y5)+(0.5,0.15)$) {\small $\vdots$};
\node[] at ($0.5*(y5)+0.5*(y6)+(0.5,0)$) {\small $Y_{t'}'$};

\node[] at ($(x_t)+(-0.6,0.1)$) {\small $x_{i+1}^{\text{t}}$};
\node[] at ($(x_b)+(-0.6,0.1)$) {\small $x_{i+1}^{\text{b}}$};
\node[] at ($(y_t)+(0.6,0.1)$) {\small $y_{i+1}^{\text{t}}$};
\node[] at ($(y_b)+(0.6,0.1)$) {\small $y_{i+1}^{\text{b}}$};




\end{tikzpicture}
\caption{Case 1 of \Cref{lem:splitting_process}. Note that trivial paths through using any two distinct edges of a star always have different lengths.}
\label{fig:split-case-1}
\end{minipage}\hfill
\begin{minipage}[t]{0.27\textwidth}
\centering
\captionsetup{width=\textwidth}
\begin{tikzpicture}[scale=0.8]

\defPt{0}{0}{x_b}
\defPt{0}{5.5}{x_t}
\defPtm{($0.075*(x_b)+0.925*(x_t)$)}{x6}
\defPtm{($0.275*(x_b)+0.725*(x_t)$)}{x5}

\defPtm{($0.425*(x_b)+0.575*(x_t)$)}{x4}
\defPtm{($0.625*(x_b)+0.375*(x_t)$)}{x3}

\defPtm{($0.725*(x_b)+0.275*(x_t)$)}{x2}
\defPtm{($0.925*(x_b)+0.075*(x_t)$)}{x1}

\defPt{3.5}{0}{y_b}
\defPt{3.5}{5.5}{y_t}

\defPtm{($0.075*(y_b)+0.925*(y_t)$)}{y6}
\defPtm{($0.275*(y_b)+0.725*(y_t)$)}{y5}

\defPtm{($0.425*(y_b)+0.575*(y_t)$)}{y4}
\defPtm{($0.625*(y_b)+0.375*(y_t)$)}{y3}

\defPtm{($0.725*(y_b)+0.275*(y_t)$)}{y2}
\defPtm{($0.925*(y_b)+0.075*(y_t)$)}{y1}

\draw[color=white, fill=black!50!,fill opacity=0.3] ($0.25*(x1)+0.75*(x2)$) -- ($0.5*(y1)+0.5*(y2)$) -- ($0.75*(y1)+0.25*(y2)$) -- ($0.5*(x1)+0.5*(x2)$) -- cycle;
\draw[color=white, fill=black!50!,fill opacity=0.3] ($0.5*(x3)+0.5*(x4)$) -- ($0.25*(y3)+0.75*(y4)$) -- ($0.5*(y3)+0.5*(y4)$) -- ($0.75*(x3)+0.25*(x4)$) -- cycle;
\draw[color=white, fill=black!50!,fill opacity=0.3] ($0.25*(x5)+0.75*(x6)$) -- ($0.5*(y5)+0.5*(y6)$) -- ($0.75*(y5)+0.25*(y6)$) -- ($0.5*(x5)+0.5*(x6)$) -- cycle;

\draw[] (x_b) \smvx;
\draw[] (x_t) \smvx;
\draw[] (y_b) \smvx;
\draw[] (y_t) \smvx;

\draw[line width= 0.75 pt] (x_b) -- (x_t);
\draw[line width= 0.75 pt] (y_b) -- (y_t);


\foreach \i in {1,...,6}
{
\draw[line width= 1 pt, dashed] (x\i) -- (y\i);
}

 \draw[line width= 1 pt,dotted] ($0.25*(x1)+0.75*(x2)$) -- ($0.5*(y1)+0.5*(y2)$);
 \draw[line width= 1 pt,dotted] ($0.5*(x1)+0.5*(x2)$) -- ($0.75*(y1)+0.25*(y2)$);
 \draw[line width= 1 pt, red] ($0.75*(x1)+0.25*(x2)$) -- ($0.25*(y1)+0.75*(y2)$);

 \draw[line width= 1 pt,red] ($0.25*(x3)+0.75*(x4)$) -- ($0.75*(y3)+0.25*(y4)$);
 \draw[line width= 1 pt,dotted] ($0.5*(x3)+0.5*(x4)$) -- ($0.25*(y3)+0.75*(y4)$);
 \draw[line width= 1 pt,dotted] ($0.75*(x3)+0.25*(x4)$) -- ($0.5*(y3)+0.5*(y4)$);

 \draw[line width= 1 pt,dotted] ($0.25*(x5)+0.75*(x6)$) -- ($0.5*(y5)+0.5*(y6)$);
 \draw[line width= 1 pt,dotted] ($0.5*(x5)+0.5*(x6)$) -- ($0.75*(y5)+0.25*(y6)$);
 \draw[line width= 1 pt,red] ($0.75*(x5)+0.25*(x6)$) -- ($0.25*(y5)+0.75*(y6)$);

\node[] at ($0.5*(x1)+0.5*(x2)+(-0.5,0)$) {\small $X_1'$};
\node[] at ($0.5*(x3)+0.5*(x4)+(-0.5,0)$) {\small $X_2'$};
\node[] at ($0.5*(x4)+0.5*(x5)+(-0.5,0.15)$) {\small $\vdots$};
\node[] at ($0.5*(x5)+0.5*(x6)+(-0.5,0)$) {\small $X_{t'}'$};

\node[] at ($0.5*(y1)+0.5*(y2)+(0.5,0)$) {\small $Y_1'$};
\node[] at ($0.5*(y3)+0.5*(y4)+(0.5,0)$) {\small $Y_2'$};
\node[] at ($0.5*(y4)+0.5*(y5)+(0.5,0.15)$) {\small $\vdots$};
\node[] at ($0.5*(y5)+0.5*(y6)+(0.5,0)$) {\small $Y_{t'}'$};

\node[] at ($(x_t)+(-0.6,0.1)$) {\small $x_{i+1}^{\text{t}}$};
\node[] at ($(x_b)+(-0.6,0.1)$) {\small $x_{i+1}^{\text{b}}$};
\node[] at ($(y_t)+(0.6,0.1)$) {\small $y_{i+1}^{\text{t}}$};
\node[] at ($(y_b)+(0.6,0.1)$) {\small $y_{i+1}^{\text{b}}$};




\end{tikzpicture}
\caption{Case 2 of \Cref{lem:splitting_process}. Edges $(x_j,y_j)$ are depicted in red. Shaded regions have many chords and will become $X_j'',Y_j''$.}
\label{fig:split-case-2}
\end{minipage}
\end{figure}

For each $j$, consider all chords in $E(X_j',Y_j')$ which interlace $(x_j,y_j)$. Observe that they can be of two types: either their $X$-vertex is above $x_j$ and their $Y$-vertex is below $y_j$, or vice versa. Consider only the interlacing chords of the more frequent type and define $X_j'',Y_j''$ to be the smallest subsections of $X'_j,Y'_j$ which contain all these chords. In particular, $x_j \notin X_j''$, $y_j \notin Y_j''$, and $(X''_j,Y''_j)$ is interlaced by $(x_j,y_j)$ as a section pair. Moreover, we still have $\sum_{j=1}^{t} e(X_j'',Y_j'') \ge \Omega(m/\log m)$ chords altogether. Consider the intervals of the form $[k,2k]$ with $k$ a power of $2$ and $k \leq m/2$. Each $e(X_j'',Y_j'')$ belongs to such an interval, of which there are $O(\log m)$. So by averaging, there exists $1 \le k \le m/2$ and a subset $S \subseteq [t]$ such that for any $j \in S$ we have $e(X_j'',Y_j'') \in [k,2k]$, and $\sum_{j \in S} e(X_j'',Y_j'') \ge \Omega(m/\log^2 m)$. Furthermore, by deleting $e(X_j'',Y_j'') - k \leq \frac{1}{2}e(X_j'',Y_j'')$ chords from each pair $X''_j,Y''_j$ with $j \in S$ (so in total at most a half of all such chords), we may assume for convenience that $e(X_j'',Y_j'')=k$ for all $j \in S$. 
Observe in particular that 
\begin{equation}\label{eq:|S|k}
|S|k \ge \Omega(m/\log^2 m).
\end{equation}

We now construct a collection of \tbpps{} for $X_j',Y_j'$ for each $j \in S$. Depending on which of the two types above was more frequent, we have two essentially symmetric cases. In the first one, $X_j''$ is above $x_j$ and $Y_j''$ is below $y_j$. 
Here we will take our path pairs to all contain the path which goes along $X_j'$ from the bottom to $x_j$, uses the chord $(x_j,y_j)$, and then goes along $Y'_j$ from $y_j$ to the top of $Y_j'$. The second path in each pair is going to follow $X_j'$ from the top until the top of $X_j''$, then follow a path inside $X_j'' \cup Y_j''$ from the top of $X_j''$ to the bottom of $Y_j''$, and then follow along $Y_j'$ to its bottom. Symmetrically, in the second case, our path pairs will all use the path through $(x_j,y_j)$ from the top of $X_j'$ to the bottom of $Y_j'$, with the other path (from the bottom of $X_j'$ to the top of $Y_j'$) passing through $X_j'' \cup Y_j''$. See \Cref{fig:two-path-pairs} for these two symmetric cases. Since the argument is analogous in both cases, from now on, let us assume we are in the first case.

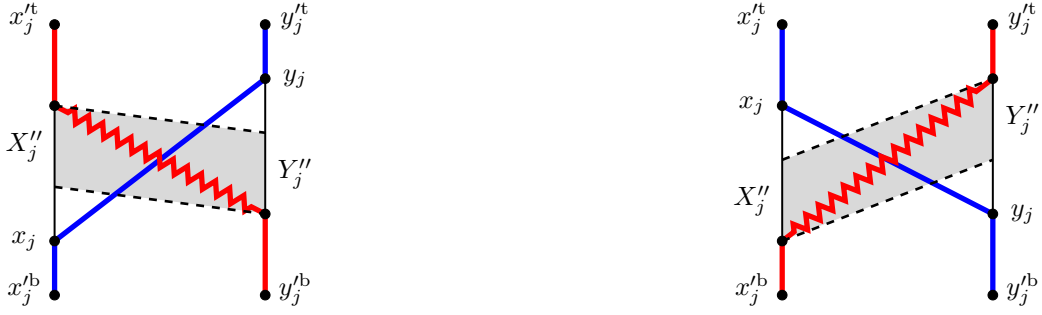
\begin{figure}
\RawFloats
\begin{minipage}[t]{0.45\textwidth}
\centering
\captionsetup{width=\textwidth}
\begin{tikzpicture}[scale=0.8]

\defPt{0}{0}{x_b}
\defPt{0}{4.5}{x_t}

\defPtm{($0.3*(x_b)+0.7*(x_t)$)}{x3}
\defPtm{($0.6*(x_b)+0.4*(x_t)$)}{x2}
\defPtm{($0.8*(x_b)+0.2*(x_t)$)}{x1}

\defPt{3.5}{0}{y_b}
\defPt{3.5}{4.5}{y_t}

\defPtm{($0.3*(y_t)+0.7*(y_b)$)}{y2}
\defPtm{($0.6*(y_t)+0.4*(y_b)$)}{y3}
\defPtm{($0.8*(y_t)+0.2*(y_b)$)}{y1}

\draw[color=white, fill=black!50!,fill opacity=0.3] (x2) -- (x3) -- (y3) -- (y2) -- cycle;

\draw[line width= 0.75 pt] (x_b) -- (x_t);
\draw[line width= 0.75 pt] (y_b) -- (y_t);

\draw[line width= 0.75 pt] (x1) -- (y1);

\draw[line width= 2 pt, blue] (y_t) -- (y1) -- (x1) -- (x_b);
\draw[line width= 2 pt, red] (x_t) -- (x3);
\draw[line width= 2 pt, red] (y_b) -- (y2);
\draw[spring, line width= 2 pt, red] (x3) -- (y2);

\foreach \i in {2,3}
{
\draw[line width= 1 pt, dashed] (x\i) -- (y\i);
}

\draw[] (x_b) \smvx;
\draw[] (x_t) \smvx;

\draw[] (x1) \smvx;
\draw[] (y1) \smvx;
\draw[] (x3) \smvx;
\draw[] (y2) \smvx;
\draw[] (y_b) \smvx;
\draw[] (y_t) \smvx;

\node[] at ($0.5*(x2)+0.5*(x3)+(-0.5,0)$) {\small $X_j''$};
\node[] at ($0.5*(y2)+0.5*(y3)+(0.5,0)$) {\small $Y_j''$};

\node[] at ($(x1)+(-0.5,0)$) {\small $x_j$};
\node[] at ($(y1)+(0.5,0)$) {\small $y_j$};
\node[] at ($(x_t)+(-0.5,0.1)$) {\small $x_j'^{\text{t}}$};
\node[] at ($(x_b)+(-0.5,0.1)$) {\small $x_j'^{\text{b}}$};
\node[] at ($(y_t)+(0.5,0.1)$) {\small $y_j'^{\text{t}}$};
\node[] at ($(y_b)+(0.5,0.1)$) {\small $y_j'^{\text{b}}$};


\end{tikzpicture}
\end{minipage}\hfill
\begin{minipage}[t]{0.45\textwidth}
\centering
\captionsetup{width=\textwidth}
\begin{tikzpicture}[scale=0.8]

\defPt{0}{0}{x_b}
\defPt{0}{4.5}{x_t}

\defPtm{($0.3*(x_b)+0.7*(x_t)$)}{x1}
\defPtm{($0.5*(x_b)+0.5*(x_t)$)}{x3}
\defPtm{($0.8*(x_b)+0.2*(x_t)$)}{x2}

\defPt{3.5}{0}{y_b}
\defPt{3.5}{4.5}{y_t}

\defPtm{($0.3*(y_t)+0.7*(y_b)$)}{y1}
\defPtm{($0.5*(y_t)+0.5*(y_b)$)}{y2}
\defPtm{($0.8*(y_t)+0.2*(y_b)$)}{y3}

\draw[color=white, fill=black!50!,fill opacity=0.3] (x2) -- (x3) -- (y3) -- (y2) -- cycle;

\draw[line width= 0.75 pt] (x_b) -- (x_t);
\draw[line width= 0.75 pt] (y_b) -- (y_t);

\draw[line width= 0.75 pt] (x1) -- (y1);

\draw[line width= 2 pt, blue] (y_b) -- (y1) -- (x1) -- (x_t);
\draw[line width= 2 pt, red] (x_b) -- (x2);
\draw[line width= 2 pt, red] (y_t) -- (y3);
\draw[spring, line width= 2 pt, red] (x2) -- (y3);

\foreach \i in {2,3}
{
\draw[line width= 1 pt, dashed] (x\i) -- (y\i);
}

\draw[] (x_b) \smvx;
\draw[] (x_t) \smvx;

\draw[] (x1) \smvx;
\draw[] (y1) \smvx;
\draw[] (x2) \smvx;
\draw[] (y3) \smvx;

\draw[] (y_b) \smvx;
\draw[] (y_t) \smvx;

\node[] at ($0.5*(x2)+0.5*(x3)+(-0.5,0)$) {\small $X_j''$};
\node[] at ($0.5*(y2)+0.5*(y3)+(0.5,0)$) {\small $Y_j''$};

\node[] at ($(x1)+(-0.5,0)$) {\small $x_j$};
\node[] at ($(y1)+(0.5,0)$) {\small $y_j$};
\node[] at ($(x_t)+(-0.5,0.1)$) {\small $x_j'^{\text{t}}$};
\node[] at ($(x_b)+(-0.5,0.1)$) {\small $x_j'^{\text{b}}$};
\node[] at ($(y_t)+(0.5,0.1)$) {\small $y_j'^{\text{t}}$};
\node[] at ($(y_b)+(0.5,0.1)$) {\small $y_j'^{\text{b}}$};

\end{tikzpicture}
\end{minipage}
\caption{Two options for how $(x_j,y_j)$ can interlace $X_j'',Y_j''$ and the corresponding \tbpps{} we use.}
\label{fig:two-path-pairs}

\end{figure}

Observe that if we find a collection of path lengths inside $X_j'' \cup Y_j''$ going from the top of $X_j''$ to the bottom of $Y_j''$, then we have the same collection of lengths of \tbpps{} for $X_j',Y_j'$ up to a possible translation by a fixed constant. This is since the \tbpps{} we consider only differ in which path they use between the top of $X_j''$ and the bottom of $Y_j''$.

It will be convenient to consider the ``flipped'' section pair $X_j'',\bar{Y_j''}$, which we obtain by reversing the path $Y_j''$ (and denote it as $\bar{Y_j''}$). In particular, the top of $\bar{Y_j''}$ is the bottom of $Y_j''$ and any top to top path in $X_j'',\bar{Y_j''}$ is just a path between the top of $X_j''$ and the bottom of $Y_j''$. 

\begin{claim}\label{claim-dichotomy}
For each $j \in S$, one of the following holds.
\begin{enumerate}    
    \item\label{itm:2} There are at least $2k^{1-(i+1)\eps}$ lengths, any two at least $k^{i\eps}$ apart, of top to top paths in $X_j'',\bar{Y_j''}$.
    \item\label{itm:3} There is an interval $I_j \subseteq \mathbb{N}$ of size $k^{i\eps}$, such that the length with respect to $X_j'',\bar{Y_j''}$ of at least $k^{(i+1)\eps}/4$ chords belongs to $I_j$.
\end{enumerate}
\end{claim}

\begin{proof}
All our chord lengths in this proof will be with respect to $X_j'',\bar{Y_j''}$. Take a minimal collection of disjoint intervals of size $k^{i\eps}$ each, which cover all chord lengths (i.e., the length with respect to $X_j'',\bar{Y_j''}$ of every chord in $E(X_j'',\bar{Y_j''})$ belongs to one of these intervals). For each such interval $I$, either $I$ can be taken as the desired $I_j$, or there must be less than $k^{(i+1)\eps}/4$ chords with length in $I$. As there are $k$ chords in $E(X_j'',\bar{Y_j''})$, this means that we needed to use at least $4k^{1-(i+1)\eps}$ intervals to cover all chord lengths. Since we took a minimal collection, each interval is non-empty. By picking one length from every second interval, we find $2k^{1-(i+1)\eps}$ chord lengths, any two separated by one of our intervals (since we took a length from every second interval). Hence, any two such lengths are at least $k^{i\eps}$ apart. For each such length, take a chord of this length and take the trivial path between the top of $X''_j$ and the top of $\bar{Y_j''}$ corresponding to this chord. This gives the $2k^{1-(i+1)\eps}$ path lengths required by item 1. 
\end{proof}

Let $S_1$ be the subset of $S$ consisting of $j$ for which case $1$ of the above claim occurred. Let us first assume this was the more common case, namely that $|S_1|\ge |S|/2$.

\begin{claim}\label{claim:far apart lengths}
If $|S_1|\ge |S|/2$, then there are $\Omega(m^{1-(i+1)\eps}/\log^2 m)$ lengths of \tbpps{} for \linebreak $X_{i+1},Y_{i+1}$ which are all at least $\Omega(m^{i\eps})$ apart. 
\end{claim}
\begin{proof}
For each $j \in S_1$, consider the $2k^{1-(i+1)\eps}$ top to top paths for $X''_j,\bar{Y_j''}$ given by item 1 of Claim \ref{claim-dichotomy}. 
Each of these paths is a path from the top of $X_j''$ to the bottom of $Y_j''$. Hence, as explained above, this gives rise to \tbpps{} for $X_j',Y_j'$ with at least 
$2k^{1-(i+1)\eps}\ge 2$ (recall that $i \le 1/\eps-1$) lengths, any two of which are at least $k^{i\eps}$ apart. Denote the set of these lengths by $L_j$. Now observe that by choosing a path-pair for $X'_j,Y'_j$ of length $\ell_{j}\in L_j$ for each $j \in S_1$, and joining these paths with subpaths of $X,Y$, we obtain a \tbpp{} for $X_{i+1},Y_{i+1}$ of length $\sum_{j \in S_1} \ell_{j}+C_0$, where $C_0$ is the sum of the lengths of the subpaths we used to join our \tbpps{} into a \tbpp{} for $X_{i+1},Y_{i+1}$; so $C_0$ is independent of our choices for $\ell_j$. \Cref{lem:sum-sets} now tells us that we can find $|S_1|k^{1-(i+1)\eps}$ lengths of \tbpps{} for $X_{i+1},Y_{i+1}$ which are at least $k^{i\eps}$ apart. By taking every $m^{i\eps}/k^{i\eps}$-th such length, we get at least $\left\lfloor |S_1|k/(m^{i\eps}/k^{i\eps}) \right\rfloor \geq
\frac12 |S_1|k^{1-\eps}/m^{i\eps}\ge
\frac14 |S|k^{1-\eps}/m^{i\eps} \ge \frac1{4} |S|k/m^{(i+1)\eps} \ge \Omega(m^{1-(i+1)\eps}/\log^2 m)$ lengths which are at least $m^{i\eps}$ apart, where the last inequality uses \eqref{eq:|S|k}. This proves the claim.
\end{proof}

Let us now use Claim \ref{claim:far apart lengths} to complete the proof in the case that $|S_1| \geq |S|/2$. 
Observe that we may combine our \tbpps{} for $X_{i+1},Y_{i+1}$ (provided by Claim \ref{claim:far apart lengths}) with top-to-top below-paths from $x_i^{\text{t}}$ to $y_i^{\text{t}}$, provided by the inductive assumption, to obtain below-paths from $x_{i+1}^{\text{t}}$ to $y_{i+1}^{\text{t}}$, which can then be extended along $X,Y$ into top-to-top paths for $X,Y$. Each such path length is of the form $\ell_1+\ell_2+C_0$, where $\ell_1$ may be chosen among our $\Omega(m^{1-(i+1)\eps}/\log^2m)$ lengths of \tbpps{} for $X_{i+1},Y_{i+1}$, which are at least $m^{i\eps}$ apart; $\ell_2$ may be chosen among $\Omega(m^{i\eps}/(C\log m)^{(2i-2)})$ lengths which belong in an interval of size $m^{i\eps}$ (as guaranteed by the induction hypothesis); and $C_0$ is independent of the choice of $\ell_1,\ell_2$. Therefore, by \Cref{lem:spread-close} we get at least $\Omega(m^{1-\eps}/(C\log m)^{2i}) \ge  \Omega(m^{1-4\eps})$ (see \eqref{eq:log-epsilons}) different lengths of top-to-top paths for $X,Y$, contradicting our assumption that the number of such lengths is smaller than $\Omega(m^{1-4\eps})$. This completes the proof in the case $|S_1| \geq |S|/2$.

So we must have the second case in Claim \ref{claim-dichotomy} occurring more often, namely that $S_2:=S \setminus S_1$ has size at least $|S|/2$. 
By definition, 
for any $j \in S_2$ there is an interval $I_j$ of size $k^{i\eps}$ such that at least $k^{(i+1)\eps}/4$ chords in $E(X''_i,\bar{Y_{j}''})$ have their length with respect to the subsection pair $X_{j}'', \bar{Y_{j}''}$ belong in $I_j$. Applying \Cref{lem:close-length chords} (with $J = I_j$), we find $\Omega(k^{\eps})$ top-to-top paths in $X_{j}'', \bar{Y_{j}''}$ whose lengths are $\Theta(k^{i\eps})$ apart. As before, this gives us \tbpps{} for $X_{j}',Y_j'$ of lengths which are $\Theta(k^{i\eps})$ apart.
Once again, we can choose one such length for each $X_j',Y_j'$ and combine the corresponding \tbpps{} into a \tbpp{} for $X_{i+1},Y_{i+1}$, with the resulting length being the sum of the lengths of the \tbpps{} we chose plus a fixed number (which is independent of our choices). \Cref{lem:sum-sets} then tells us that we can find \tbpps{} for $X_{i+1},Y_{i+1}$ of $\Omega(|S_2|k^{\eps})\ge \Omega(|S|k^{\eps})$ lengths which are $\Theta(k^{i \eps})$ apart, say between $c_1 k^{i \eps}$ and $c_2 k^{i \eps}$. Once again we only consider every $m^{i\eps}/c_1k^{i\eps}$-th length, 
thus guaranteeing that every two such lengths are at least 
$m^{i\eps}$ and at most $\frac{c_2}{c_1}m^{i\eps} =   O(m^{i\eps})$ apart. 
The number of these lengths is at least $c_1k^{i\eps}/m^{i\eps}\cdot \Omega(|S|k^{\eps})\ge \Omega(k^{(i+1)\eps-1} m^{1-i\eps}/\log^2 m)\ge \Omega(m^{\eps}/\log^2 m)$, with the first inequality using \eqref{eq:|S|k}. Take the first $\Theta(m^{\eps}/\log^2 m)$ of them.
Since these lengths are at most $O(m^{i\eps})$ apart, they are all contained in an interval of size $O(m^{(i+1)\eps})$.
Now similarly as in the previous case, we can combine any of these \tbpps{} with our below paths from $x_i^{\text{t}}$ to $y_i^{\text{t}}$, which are provided by the induction hypothesis. \Cref{lem:spread-close} ensures that by doing so, we get at least 
$\Omega(m^{\eps}/\log^2 m) \cdot \frac{m^{i\eps}}{6(C\log m)^{2(i-1)}} = 
\Omega\left(\frac{m^{(i+1)\eps}}{C^{2i-2}(\log m)^{2i}}\right)$ lengths which all belong to an interval of size $O(m^{(i+1)\eps})$. By potentially losing a constant proportion of the lengths, we can guarantee they all lie in an interval of size at most $m^{(i+1)\eps}$, and by choosing $C$ to be large enough, we will have at least $\frac{m^{(i+1)\eps}}{(C\log m)^{2i}}$ of them,  completing the induction step and hence \nolinebreak the \nolinebreak proof.
\end{proof}

We are now in a position to prove Theorem \ref{thm:main}, which we rephrase in the following quantitative form:

\begin{thm}\label{thm:main-with-numbers}
Every $n$-vertex Hamiltonian graph with minimum degree $3$ has cycle spectrum of size at least $\Omega\Big({n}/{2^{6\sqrt{\log n \log \log n}}}\Big).$
\end{thm}

\begin{proof}

Let us fix a Hamilton cycle $H$ in our graph and a direction for this cycle. We define the length of a chord $\{x,y\}$ to be the length of the shorter of the two paths along the cycle between $x$ and $y$; in particular, the length is always at most $n/2$. 
For every vertex $v$, we fix a single chord $c(v)$ incident to $v$. By the pigeonhole principle, there must exist an $\ell\le n/4$ such that at least $n/\log n$ vertices have the length of their chosen chord in $[\ell,2\ell]$. Now let us partition $H$ into consecutive sections $B_1,\ldots, B_s,R$ (appearing in this order along $H$), where $|B_s|=4\ell$ and $s= \floor{ \frac{n}{4\ell}} \ge 1$, and the starting vertex for $B_1$ is chosen uniformly at random among the vertices of the cycle. We will talk of the first half of $B_j$ (which is just the set of the first $2\ell$ vertices of $B_j$), the first quarter of $B_j$ (the set of the first $\ell$ vertices), etc.

A vertex $v$ is said to be \textit{good} if $c(v)$ is contained in some $B_j$ and if the two endpoints of $c(v)$ belong to different halves of $B_j$. 
We claim that the probability that a given $v$ with $c(v) \in [\ell,2\ell]$ is good is at least $\frac{1}{8}$. To see this, write $c(v) = \{v,w\}$, and consider the shorter path (along the cycle) between $v$ and $w$; recall that the length of $c(v)$ is defined as the length of this path. There are two cases: if when walking on this path from $v$ to $w$, we walk in the direction of the cycle (which was fixed at the beginning), then $v$ is guaranteed to be good if it belongs to the second quarter of some $B_j$. Otherwise, i.e. if we walk against the direction of the cycle, then $v$ is guaranteed to be good if it belongs to the third quarter of some $B_j$. In any case, the probability that $v$ is good is at least $\frac{1}{4} \cdot \frac{|B_1| + \ldots + |B_s|}{n}$. Since $|B_1| + \ldots + |B_s| \geq n/2$, this probability is at least $\frac{1}{8}$, as claimed. 
By linearity of expectation, the expected number of good vertices is at least $\frac{n}{8 \log n}$. Let us now fix an outcome $B_1,\ldots, B_s,R$ for which we have at least this many good vertices.
Let $I$ be the set of $i \in [s]$ such that $B_i$ contains at least $\frac{\ell}{4\log n}$ good vertices. Observe that the sections $B_i$, $i \notin I$, contribute in total at most $s\cdot \frac{\ell}{4\log n} \le \frac{n}{16\log n}$ to the overall number of good vertices. Therefore, at least $\frac{n}{16 \log n}$ good vertices belong to sections $B_i$ with $i \in I$. In particular, $|I| \geq \frac{n}{64 \ell \log n}$, since each $B_i$ has $4\ell$ vertices.

For each $i \in I$, we will find inside $B_i$ paths of at least $1+ \Omega(\ell/2^{5\sqrt{\log n \log \log n}})$ different lengths joining the endpoints of $B_i$. Let us first complete the proof using this. To do so, choose for each $i \in I$, one of these paths joining the endpoints of $B_i$, and combine these paths into a cycle using $H$. The length of the resulting cycle is equal to the sum of the lengths of the paths we chose plus a fixed number, independent of our choices (i.e. the total length of the pieces of $H$ we used to join the paths). So by \Cref{lem:sum-sets}, these cycles take at least $\Omega\left(\frac{n}{ \ell \log n}\cdot {\ell}/{2^{5\sqrt{\log n \log \log n}}}\right) \ge \Omega\left({n}/{2^{6\sqrt{\log n \log \log n}}}\right)$ different lengths, as desired.

What remains to be proved is that for $i \in I$, $B_i$ contains paths of $1+ \Omega(\ell/2^{5\sqrt{\log n \log \log n}})$ different lengths joining its endpoints. 
By definition, $B_i$ has at least $\frac{\ell}{4\log n}$ good vertices.
Since $B_i$ contains at least one chord, we know it has at least two path lengths (one coming from the length of $B_i$ itself and the other from the path only using a single chord). In particular, we may assume $\ell \gg 2^{5\sqrt{\log n \log \log n}}$ or we are done. 
Let $X$ denote the first half of $B_i$ and $Y$ its second half. We think of $X,Y$ as a section pair (omitting the middle edge of $B_i$), and consider the endpoints of $B_i$ to be the top vertices of $X$ and $Y$.
By definition, every chord $c(v)$ which corresponds to a good vertex $v$ has one endpoint in $X$ and one in $Y$. Let us assume without loss of generality that $X$ contains at least half (so at least $\frac{\ell}{8\log n}$) of all good vertices in $B_i$. Delete any chord which was not chosen by one of the good vertices in $X$. This way, we make sure that every vertex in $X$ is incident to at most one chord. Also, at least $\frac{\ell}{8\log n}$ chords remain. 

Let us now apply \Cref{lem:initial_split} to this section-pair $X,Y$ with parameter $m=\frac{\ell}{8\log n}$ and $k=\sqrt{\log n}$, and thus obtain either a vertex $v \in Y$ incident to at least $\Omega\left({\ell}/{\log^2 n}\right)$ chords, or an interlacing or parallel collection of $k$ subsection pairs $X_1,Y_1;\ldots; X_k,Y_k$, each containing at least $\Omega\left({\ell}/{\log^3 n}\right)$ chords.  

Suppose first that there is a vertex $v \in Y$ incident to at least $\Omega\left({\ell}/{\log^2 n}\right)$ chords $(v,x)$, $x \in X$. Observe that all these chords have different lengths. Now, for each such chord, consider the path between the endpoints of $B_i$, which uses only this chord and no others. 
These are paths of $\Omega\left({\ell}/{\log^2 n}\right)$ different lengths between the endpoints of $B_i$, as required. 

Suppose now that we are in the second case, namely that there exists an interlacing or parallel collection $X_1,Y_1;\ldots; X_k,Y_k$ of subsection pairs with $e(X_i,Y_i) \geq \Omega\left({\ell}/{\log^3 n}\right)$ for each $i = 1,\dots,k$. If these subsection pairs are parallel, then we may immediately apply \Cref{lem:main-part} with $t=k$ and $m=\Omega({\ell}/{\log^3 n})$ to get the desired number $\Omega\big( \frac{\ell}{\log^3n}/2^{4\sqrt{\log n \log \log n}} \big) \geq \Omega(\ell/2^{5\sqrt{\log n \log \log n}})$ of different path lengths between the endpoints of $B_i$. 
In the case that $X_1,Y_1;\ldots; X_k,Y_k$ are interlacing, we need an extra step, as follows. Suppose without loss of generality that $X_i$ is below $X_j$ (and hence $Y_i$ is above $Y_j$) for each $1 \leq i < j \leq k$. Take $(x_1,y_1) \in E(X_1,Y_1)$ and $(x_2,y_2) \in E(X_2,Y_2)$. 
Let us now ``reroute'' $B_i$ along the interlacing chords $(x_1,y_1)$,$(x_2,y_2)$; in other words, we consider a new path $B_i'$, obtained by following $X$ from the top until $x_2$, jumping along $(x_2,y_2)$, following $B_i$ (now in the other direction) until $x_1$, jumping along $(x_1,y_1)$, and finally following $Y$ until its top. See \Cref{fig:rerouting} for a picture. 
Note that 
$B_i'$ splits into the section-pair $X' := X[x^{\text{t}},x_2]$, $Y' := B_i[y_2,x_1],(x_1,y_1),Y[y_1,y^{\text{t}}]$, with $X'$ having the same top as $X$ and $Y'$ the same top as $Y$. Observe that $X_k,\ldots, X_3$ still appear in this order along $X'$, but, crucially, $Y_3,\dots,Y_k$ now appear in the reverse order, see \Cref{fig:0.9}. In other words, the subsection pairs $X_3,Y_3;\ldots; X_k,Y_k$ are parallel in $X',Y'$. 
So we may now apply \Cref{lem:main-part} with $t=k-2$ and $m=\Omega({\ell}/{\log^3 n})$ to get the desired number of different path lengths between the endpoints of $B_i$. This completes the proof.  
\begin{figure}
\RawFloats
\begin{minipage}[t]{0.45\textwidth}
\centering
\captionsetup{width=\textwidth}
\begin{tikzpicture}[scale=1]
\defPt{0}{0}{x_b}
\defPt{0}{4.5}{x_t}
\defPtm{($0.05*(x_b)+0.95*(x_t)$)}{x10}
\defPtm{($0.15*(x_b)+0.85*(x_t)$)}{x9}

\defPtm{($0.25*(x_b)+0.75*(x_t)$)}{x8}
\defPtm{($0.35*(x_b)+0.65*(x_t)$)}{x7}

\defPtm{($0.45*(x_b)+0.55*(x_t)$)}{x6}
\defPtm{($0.55*(x_b)+0.45*(x_t)$)}{x5}

\defPtm{($0.65*(x_b)+0.35*(x_t)$)}{x4}
\defPtm{($0.75*(x_b)+0.25*(x_t)$)}{x3}

\defPtm{($0.85*(x_b)+0.15*(x_t)$)}{x2}
\defPtm{($0.95*(x_b)+0.05*(x_t)$)}{x1}

\defPt{4.5}{0}{y_b}
\defPt{4.5}{4.5}{y_t}

\defPtm{($0.05*(y_b)+0.95*(y_t)$)}{y2}
\defPtm{($0.15*(y_b)+0.85*(y_t)$)}{y1}

\defPtm{($0.25*(y_b)+0.75*(y_t)$)}{y4}
\defPtm{($0.35*(y_b)+0.65*(y_t)$)}{y3}

\defPtm{($0.45*(y_b)+0.55*(y_t)$)}{y6}
\defPtm{($0.55*(y_b)+0.45*(y_t)$)}{y5}

\defPtm{($0.65*(y_b)+0.35*(y_t)$)}{y8}
\defPtm{($0.75*(y_b)+0.25*(y_t)$)}{y7}

\defPtm{($0.85*(y_b)+0.15*(y_t)$)}{y10}
\defPtm{($0.95*(y_b)+0.05*(y_t)$)}{y9}

\draw[color=white, fill=black!50!,fill opacity=0.3] (x1) -- (x2) -- (y2) -- (y1) -- cycle;
\draw[color=white, fill=black!50!,fill opacity=0.3] (x3) -- (x4) -- (y4) -- (y3) -- cycle;
\draw[color=white, fill=black!50!,fill opacity=0.3] (x5) -- (x6) -- (y6) -- (y5) -- cycle;
\draw[color=white, fill=black!50!,fill opacity=0.3] (x7) -- (x8) -- (y8) -- (y7) -- cycle;
\draw[color=white, fill=black!50!,fill opacity=0.3] (x9) -- (x10) -- (y10) -- (y9) -- cycle;

\draw[line width= 0.75 pt] (x_b) -- (x_t);
\draw[line width= 0.75 pt] (y_b) -- (y_t);



\foreach \i in {1,...,10}
{
\draw[line width= 1 pt, dashed] (x\i) -- (y\i);
}




 \draw[line width= 2 pt, red] ($0.5*(x3)+0.5*(x4)$) -- ($0.5*(y3)+0.5*(y4)$);
\draw[line width= 2 pt, red] ($0.5*(x1)+0.5*(x2)$) -- ($0.5*(y1)+0.5*(y2)$);
\draw[line width= 2 pt, red] (x_b) -- ($0.5*(x1)+0.5*(x2)$);
\draw[line width= 2 pt, red] (x_t) -- ($0.5*(x3)+0.5*(x4)$);
\draw[line width= 2 pt, red] (y_t) -- ($0.5*(y1)+0.5*(y2)$);
\draw[line width= 2 pt, red] (y_b) -- ($0.5*(y3)+0.5*(y4)$);

\draw[line width= 2 pt, red] (x_b) to[out=-30, in=-150] (y_b);

\node[] at ($0.5*(x1)+0.5*(x2)+(-0.5,0)$) {\small $x_1$};
\node[] at ($0.5*(x3)+0.5*(x4)+(-0.5,0)$) {\small $x_2$};
\node[] at ($0.5*(x5)+0.5*(x6)+(-0.5,0)$) {\small $X_3$};
\node[] at ($0.5*(x7)+0.5*(x8)+(-0.5,0)$) {\small $X_4$};
\node[] at ($0.5*(x9)+0.5*(x10)+(-0.5,0)$) {\small $X_5$};

\node[] at ($0.5*(y1)+0.5*(y2)+(0.5,0)$) {\small $y_1$};
\node[] at ($0.5*(y3)+0.5*(y4)+(0.5,0)$) {\small $y_2$};
\node[] at ($0.5*(y5)+0.5*(y6)+(0.5,0)$) {\small $Y_3$};
\node[] at ($0.5*(y7)+0.5*(y8)+(0.5,0)$) {\small $Y_4$};
\node[] at ($0.5*(y9)+0.5*(y10)+(0.5,0)$) {\small $Y_5$};

\node[] at ($(x_t)+(-0.5,0.1)$) {\small $x^{\text{t}}$};
\node[] at ($(x_b)+(-0.5,-0.1)$) {\small $x^{\text{b}}$};
\node[] at ($(y_t)+(0.5,0.1)$) {\small $y^{\text{t}}$};
\node[] at ($(y_b)+(0.5,-0.1)$) {\small $y^{\text{b}}$};

\draw[] (x_b) \smvx;
\draw[] (x_t) \smvx;
\draw[] (y_b) \smvx;
\draw[] (y_t) \smvx;

\draw[] ($0.5*(y1)+0.5*(y2)$) \smvx;
\draw[] ($0.5*(y3)+0.5*(y4)$) \smvx;
\draw[] ($0.5*(x1)+0.5*(x2)$) \smvx;
\draw[] ($0.5*(x3)+0.5*(x4)$) \smvx;

\end{tikzpicture}
\caption{Rerouting through a path $B'_i$ marked in red.}
\label{fig:rerouting}
\end{minipage}\hfill
\begin{minipage}[t]{0.45\textwidth}
\centering
\captionsetup{width=\textwidth}
\begin{tikzpicture}[scale=1]

\defPt{0}{-1.66}{x_b}
\defPt{0}{2.34}{x_t}
\defPtm{($0.05*(x_b)+0.95*(x_t)$)}{x10}
\defPtm{($0.15*(x_b)+0.85*(x_t)$)}{x9}

\defPtm{($0.25*(x_b)+0.75*(x_t)$)}{x8}
\defPtm{($0.35*(x_b)+0.65*(x_t)$)}{x7}

\defPtm{($0.45*(x_b)+0.55*(x_t)$)}{x6}
\defPtm{($0.55*(x_b)+0.45*(x_t)$)}{x5}

\defPtm{($0.65*(x_b)+0.35*(x_t)$)}{x4}
\defPtm{($0.75*(x_b)+0.25*(x_t)$)}{x3}

\defPtm{($0.85*(x_b)+0.15*(x_t)$)}{x2}
\defPtm{($0.95*(x_b)+0.05*(x_t)$)}{x1}

\defPt{4.5}{-0.3}{y_b}
\defPt{4.5}{3.6}{y_t}

\defPtm{($0.05*(y_b)+0.95*(y_t)$)}{y2}
\defPtm{($0.15*(y_b)+0.85*(y_t)$)}{y1}

\defPtm{($0.25*(y_b)+0.75*(y_t)$)}{y4}
\defPtm{($0.35*(y_b)+0.65*(y_t)$)}{y3}

\defPtm{($0.45*(y_b)+0.55*(y_t)$)}{y6}
\defPtm{($0.55*(y_b)+0.45*(y_t)$)}{y5}

\defPtm{($0.65*(y_b)+0.35*(y_t)$)}{y8}
\defPtm{($0.75*(y_b)+0.25*(y_t)$)}{y7}

\defPtm{($0.85*(y_b)+0.15*(y_t)$)}{y10}
\defPtm{($0.95*(y_b)+0.05*(y_t)$)}{y9}

\draw[color=white, fill=black!50!,fill opacity=0.3] (x5) -- (x6) -- (y10) -- (y9) -- cycle;
\draw[color=white, fill=black!50!,fill opacity=0.3] (x7) -- (x8) -- (y8) -- (y7) -- cycle;
\draw[color=white, fill=black!50!,fill opacity=0.3] (x9) -- (x10) -- (y6) -- (y5) -- cycle;

\draw[line width= 0.75 pt] (y_b) -- (y_t);



\foreach \i in {7,8}
{
\draw[line width= 1 pt, dashed] (x\i) -- (y\i);
}
\draw[line width= 1 pt, dashed] (x5) -- (y9);
\draw[line width= 1 pt, dashed] (x6) -- (y10);
\draw[line width= 1 pt, dashed] (x9) -- (y5);
\draw[line width= 1 pt, dashed] (x10) -- (y6);




\draw[line width= 2 pt, red] (x_t) -- ($0.5*(x3)+0.5*(x4)$);
\draw[line width= 2 pt, red] (y_t) -- ($0.5*(y1)+0.5*(y2)$);
\draw[line width= 2 pt, red] (y_b) -- ($0.5*(y3)+0.5*(y4)$);
\draw[line width= 2 pt, red] ($(y_t)+(0,0.4)$) -- ($0.5*(y3)+0.5*(y4)$);

\draw[line width= 2 pt, red] ($0.5*(x3)+0.5*(x4)$) to[out=-30, in=-150] (y_b);

\node[] at ($0.5*(x3)+0.5*(x4)+(-0.5,0)$) {\small $x_2$};
\node[] at ($0.5*(x5)+0.5*(x6)+(-0.5,0)$) {\small $X_3$};
\node[] at ($0.5*(x7)+0.5*(x8)+(-0.5,0)$) {\small $X_4$};
\node[] at ($0.5*(x9)+0.5*(x10)+(-0.5,0)$) {\small $X_5$};

\node[] at ($0.5*(y1)+0.5*(y2)+(0.5,0)$) {\small $x_1$};
\node[] at ($0.5*(y3)+0.5*(y4)+(0.5,0)$) {\small $y^{\text{b}}$};
\node[] at ($0.5*(y5)+0.5*(y6)+(0.5,0)$) {\small $Y_5$};
\node[] at ($0.5*(y7)+0.5*(y8)+(0.5,0)$) {\small $Y_4$};
\node[] at ($0.5*(y9)+0.5*(y10)+(0.5,0)$) {\small $Y_3$};

\node[] at ($(x_t)+(-0.5,0.1)$) {\small $x^{\text{t}}$};
\node[] at ($0.5*(y2)+0.5*(y3)+(0.5,0)$) {\small $x^{\text{b}}$};
\node[] at ($(y_t)+(0.5,0)$) {\small $y_1$};
\node[] at ($(y_t)+(0.5,0.5)$) {\small $y^{\text{t}}$};
\node[] at ($(y_b)+(0.5,-0.1)$) {\small $y_2$};

\draw[] (x_t) \smvx;
\draw[] (y_b) \smvx;
\draw[] (y_t) \smvx;

\draw[] ($0.5*(y1)+0.5*(y2)$) \smvx;
\draw[] ($0.5*(y3)+0.5*(y4)$) \smvx;
\draw[] ($0.5*(x3)+0.5*(x4)$) \smvx;
\draw[] ($0.5*(y2)+0.5*(y3)$) \smvx;
\draw[] ($(y_t)+(0,0.4)$) \smvx;


\end{tikzpicture}
\caption{Rerouted path $B'_i$ with new relative positions of the section pairs }
\label{fig:0.9}
\end{minipage}
\end{figure}
\end{proof}

\section{Concluding remarks}

In this paper, we proved that an $n$-vertex Hamiltonian graph of minimum degree $3$ has cycles of 
$n^{1-o(1)}$ different lengths, which shows that Conjecture \ref{conj:main} holds asymptotically. Moreover, for the original question of Jacobson and Lehel (dealing with graphs of bounded degree), 
we can use our ideas to get a better quantitative bound of $\frac{n}{\text{polylog}(n)}$. Still, it would be very interesting to prove a linear bound on the number of cycles, 
even in the $3$-regular case. Towards this, we propose the following natural intermediate steps:

\begin{conj}\label{conj:sqrt}
In an $n$-vertex Hamiltonian graph with minimum degree $3$, there are: 
\begin{enumerate}
    \item $\Omega(\sqrt{n})$ cycle lengths all belonging to an interval of size $O(\sqrt{n})$.
    \item $\Omega(\sqrt{n})$ cycle lengths any two of which are at least  $\Omega(\sqrt{n})$ apart.
\end{enumerate}
\end{conj}
Observe first that Conjecture \ref{conj:main} immediately implies Conjecture \ref{conj:sqrt}. 
On the other hand, a slight strengthening of Conjecture \ref{conj:sqrt} already implies Conjecture \ref{conj:main}. Indeed, suppose that in Conjecture \ref{conj:sqrt}, we replace the assumption of minimum degree $3$ with the assumption that linearly many vertices have degree at least $3$. Furthermore, suppose that instead of starting with a Hamilton cycle we start with a Hamilton path, and instead of cycle lengths we consider path lengths of paths going between the endpoints of the given Hamilton path. 
Then by splitting our Hamilton cycle (or a ``rerouted'' cycle) into two paths (similarly to what is done in the proof of Theorem \ref{thm:main-with-numbers}), we can apply this strengthened version of Conjecture \ref{conj:sqrt} and obtain linearly many cycle lengths (see Section \ref{subsec:sketch} for some details). 

Finally, let us also discuss the following related question, which was already mentioned in passing in the introduction. What is the minimum $C$ such that every $n$-vertex Hamiltonian graph with minimum degree $3$ contains a second cycle of length at least $n - C$? Gir{\~a}o, Kittipassorn and Narayanan \cite{GKN} show that $C = O(n^{4/5}) = o(n)$, and conjecture that actually $C = O(1)$. It is also interesting to ask whether the answer changes if we want in addition to have $C>0$ (i.e., if we forbid this second cycle from being a Hamilton cycle). Such a result might be useful to attack the conjecture of Jacobson and Lehel.


\providecommand{\bysame}{\leavevmode\hbox to3em{\hrulefill}\thinspace}
\providecommand{\MR}{\relax\ifhmode\unskip\space\fi MR }
\providecommand{\MRhref}[2]{%
  \href{http://www.ams.org/mathscinet-getitem?mr=#1}{#2}
}
\providecommand{\href}[2]{#2}

\end{document}